\documentclass{amsart}

\usepackage{enumerate, amsmath, amsfonts, amssymb, amsthm, wasysym, xcolor,
url, hyperref, hypcap, frcursive,comment,stmaryrd}
\hypersetup{colorlinks=true, citecolor=darkblue, linkcolor=darkblue}

\usepackage[small]{caption}

\usepackage{listings,multicol,supertabular,tabularx,caption,graphicx}
\newcolumntype{v}{>{\centering}p}

\author{Guillaume Chapuy}
\address{CNRS, LIAFA, Universit\'e Paris Diderot -- Paris 7}
\email{guillaume.chapuy@liafa.univ-paris-diderot.fr}
\urladdr{http://www.liafa.univ-paris-diderot.fr/\~{ }chapuy/}

\author{Christian Stump}
\address{Institut f\"ur Algebra, Zahlentheorie, Diskrete Mathematik, Universit\"at Hannover}
\email{stump@math.uni-hannover.de}
\urladdr{http://homepage.univie.ac.at/christian.stump/}

\newcommand{\step}[1]{\medskip\noindent\textbf{Step #1.}} 

\newcommand{\wprod}{\wr}
\newcommand{\hook}[2]{\mathfrak{h}_{#2}^{#1}}
\newcommand{\qhook}[2]{\mathfrak{qh}_{#2}^{#1}}
\newcommand{\defn}[1]{\emph{\darkblue #1}} 
\newcommand{\GL}{\operatorname{GL}} 
\newcommand{\Sym}{\operatorname{Sym}} 
\newcommand{\CC}{\mathbb{C}} 
\newcommand{\Sn}{\mathcal{S}} 
\newcommand{\formula}{{\operatorname{FAC}}} 

\newcommand{\fixspace}{\operatorname{Fix}} 
\renewcommand{\ker}{{\operatorname{ker}}} 
\newcommand{\one}{{1\!\!1}} 
\newcommand{\occ}[1]{\operatorname{occ}(#1)} 
\newcommand{\trace}{\mathrm{Tr}} 

\newcommand{\content}{\operatorname{Content}} 
\renewcommand{\mod}{\operatorname{mod}} 

\newcommand{\ghook}[3]{{}_#3\hook{#1}{#2}} 
\newcommand{\classR}{\mathcal{R}}
\newcommand{\classCox}{\mathcal{C}}

\newcommand{\parA}{[n],\dots}
\newcommand{\parB}{[1^n],\dots}
\newcommand{\parC}{\qhook{n}{k},\dots}
\newcommand{\parD}{\hook{n-1}{k},\dots,1,\dots}

\newcommand{\bigmultiset}[1]{\big\{\!\!\big\{ #1 \big\}\!\!\big\}} 

\newcommand{\Hgroup}{K}
\newcommand{\h}{\kappa}

\definecolor{darkblue}{rgb}{0,0,0.7} 
\newcommand{\darkblue}{\color{darkblue}} 

\newtheorem{theorem}{Theorem}[section] 
\newtheorem{lemma}[theorem]{Lemma}     
\newtheorem{corollary}[theorem]{Corollary}
\newtheorem{proposition}[theorem]{Proposition}

\theoremstyle{definition}

\newtheorem{remark}{Remark}

\begin{document}
~\vspace{-1cm}
\title[Counting factorizations of Coxeter elements]{Counting factorizations of Coxeter elements\\ into products of reflections}
\thanks{
2000 Mathematics Subject Classification 20F55 (primary), 05E10 (secondary).
\\
We first considered the question addressed in this paper during the \emph{Formal Power Series and Algebraic Combinatorics} conference in Nagoya, Japan, in August 2012. We are
grateful to the organizers for this wonderful meeting, and C.S. acknowledges the financial
support.\\
\indent G.C. acknowledges partial support from the ERC grant StG 208471 -- \emph{ExploreMaps}, the grant 
grant ANR-08-JCJC-0011 -- \emph{IComb} and
the grant ANR-12-JS02-001-01 -- \emph{Cartaplus}.\\
\indent C.S. acknowledges support
from the German Research Foundation DFG, grant STU 563/2-1
``Coxeter-Catalan combinatorics''.\\
\indent We thank Christian Krattenthaler, Vic Reiner, and Vivien Ripoll for very stimulating discussions (in particular for their contributions to Remark~\ref{rem:automorphism}), 
Eduard Looijenga for providing us with a copy of~\cite{Deligne},
and Gunter Malle and Nathan Reading for valuable comments on previous versions of this paper.
}

\maketitle

\begin{abstract}
  In this paper, we count factorizations of Coxeter elements in well-generated complex reflection groups into products of reflections.
  We obtain a simple product formula for the exponential generating function of such factorizations, which is expressed uniformly in terms of natural parameters of the group.
  In the case of factorizations of minimal length, we recover a formula due to P.~Deligne, J.~Tits and D.~Zagier in the real case and to D.~Bessis in the complex case.
  For the symmetric group, our formula specializes to a formula of D.~M.~Jackson.
\end{abstract}

\vspace*{-15pt}
\tableofcontents

\vspace*{-20pt}
\section{Introduction}

One of the many equivalent forms of Cayley's formula~\cite{Cay1889} counting
labeled trees asserts that the number of factorizations of the long cycle
$(1,2,\dots,N)$ as a product of $(N-1)$ transpositions is \begin{align*}
\#\big\{ \tau_1 \tau_2 \cdots \tau_{N-1} = (1,2,\ldots,N) \big\} = N^{N-2}.
\end{align*} A widely considered generalization of this formula comes from the
field of map enumeration.  Instead of considering minimal factorizations, one
can as well consider \emph{higher genus factorizations}, which are
factorizations of the long cycle into $N-1+2g$ transpositions.  
The term
\emph{genus} and the letter $g$ come from the fact that this $g$ is indeed the
genus of a surface associated to a natural embedded graph
representing the factorization, see for
example~\cite{LandoZvonkine} for further details.  The following exponential
generating function identity, proved in~\cite[Corollary~4.2]{Jac1988}, provides a beautiful and
compact way to count such
factorizations,
\begin{align*} \sum_{g\geq 0}
\frac{t^{N-1+2g}}{(N-1+2g)!}\#\Big\{ \tau_1
\tau_2\dots\tau_{N-1+2g}=(1,2,\dots,N)\Big\} =
\frac{1}{N!}\left(e^{t\frac{N}{2}}-e^{-t\frac{N}{2}}\right)^{N-1}.
\end{align*} Observe that
near $t=0$ this formula has the expansion $\frac{t^{N-1}}{(N-1)!} N^{N-2}$, in
agreement with Cayley's formula.

\bigskip

Another way to generalize Cayley's formula is to count factorizations of Coxeter elements in real reflection groups, and, even more generally, in well-generated complex reflection groups.
If we replace the symmetric group by any well-generated complex reflection group $W$ of rank~$n$ with Coxeter number $h$, the number of factorizations of a fixed Coxeter element $c \in W$ into a product of~$n$ reflections is given by the formula
\begin{align}\label{eq:DTZ}
  \#\big\{ \tau_1 \tau_2 \dots \tau_n = c \big\} = \frac{n!}{|W|}h^n.
\end{align}
In the case of real reflection groups, this formula was proved in the early 1970's in a letter from P.~Deligne to E.~Looijenga~\cite{Deligne} crediting discussions with J.~Tits and D.~Zagier. The remaining cases were then proved by D.~Bessis in~\cite[Proposition~7.5]{Bes2007}, using the geometry of braid groups. Formula~\eqref{eq:DTZ} was as well rediscovered for real reflection groups in the context of the noncrossing partition lattice, see~\cite[Proposition~9]{Chapoton}, \cite[Corollary~3.6.10]{Armstrong}, and \cite[Theorem~3.6]{Reading}. 

For the symmetric group $\Sn_N$ the rank is given by $n=N-1$ and the Coxeter number is $h=N$, we thus again get back Cayley's formula.

\medskip

In this paper, we provide a uniform generalization of both results by counting \lq\lq higher genus\rq\rq\ factorizations of Coxeter elements in well-generated complex reflection groups into products of reflections. The main result is the following identity for their exponential generating function.
We refer to Section~\ref{sec:background} for the definitions used here.
\begin{theorem}\label{thm:main}
Let $W$ be an irreducible well-generated complex reflection group of rank~$n$.
Let $c$ be a Coxeter element in $W$, let $\classR$ be the set of all reflections in~$W$, and let $\classR^*$ be the set of all reflecting hyperplanes. Define
\begin{align*}
  \formula_W(t) := \sum_{\ell\geq 0} \frac{t^\ell}{\ell!}
    \#\Big\{ (\tau_1,\tau_2,\dots,\tau_\ell)\in \mathcal{R}^\ell\ ,\ \tau_1 \tau_2 \dots \tau_\ell = c \Big\}
\end{align*}
to be the exponential generating function of factorizations of $c$ into a product of reflections.
Then $\formula_W(t)$ is given by the formula
\begin{align}
  \formula_W(t)  =
    \frac{1}{|W|}\left(e^{t|\classR|/n}-e^{-t|\classR^*|/n}\right)^{n}. \label{eq:coxetergenus}
\end{align}
\end{theorem}
The Coxeter number $h$ is given by $(|\classR|+|\classR^*|)/n$, see Formula~\eqref{eq:nhRR} in Section~\ref{sec:background} below. Therefore the expansion of \eqref{eq:coxetergenus} near $t=0$ gives back Formula~\eqref{eq:DTZ}. Moreover, in the case of real reflection groups one has $|\classR|=|\classR^*|$, so the exponents in the righthand side of~\eqref{eq:coxetergenus} are given by $th/2$ and $-th/2$. For real reflection groups, Theorem~\ref{thm:main} has thus the following form.
\begin{corollary}\label{cor:Coxeter}
  Let $W$ be an irreducible real reflection group of rank~$n$ and Coxeter number $h$. 
  Then $\formula_W(t)$ is given by the formula
  \begin{align}\label{eq:corollary}
  \formula_W(t) = \frac{1}{|W|}\left(e^{th/2}-e^{-th/2}\right)^{n}.
  \end{align}
\end{corollary}

\begin{remark}
\label{rem:automorphism}
  There are several conventions in the literature for the definition of Coxeter elements in well-generated complex reflection groups. 
  For the definition we use here, there might be several conjugacy classes $\classCox_\zeta$ of Coxeter elements for the various primitive $h$-th roots of unity $\zeta$. One of these conjugacy classes corresponds to the Coxeter elements in the sense of the more restrictive definition used for example in~\cite{Bes2007}. Notice that, strictly speaking, the definition of $\formula_W(t)$ depends a priori on the conjugacy class of the chosen Coxeter element.
  Therefore, Theorem~\ref{thm:main} implies as well that the number of factorizations is independent of this choice, which is why we suppress the information of the conjugacy class in the notation $\formula_W(t)$. 
  This seems to illustrate a more general phenomenon, namely that enumerative properties of Coxeter elements seem to be valid for the more general definition.
  For the infinite families of matrix groups defined in Section~\ref{sec:background}, this phenomenon can easily be explained by the existence of automorphisms of the group, inherited from automorphisms of the base field $\mathbb{Q}[\zeta]$, that preserve reflections and send one class of Coxeter elements to another. For exceptional groups, however, we do not know of such a simple explanation\footnote{In the same vein, we observe that the various choices of Coxeter elements lead to isomorphic non-crossing partition lattices (see e.g.~\cite{BR2007} for the definition). For the infinite families, this follows from the same argument as above, whereas we have used a computer to verify this fact for the exceptional groups.}.
  We refer the reader to Section~\ref{sec:background} for more details on our definition and for a comparison with other definitions.
  Finally, we want to emphasize again that Theorem~\ref{thm:main} is \emph{in particular} true with the more restrictive definition of Coxeter elements.
\end{remark}
\begin{remark}
  There is no loss of generality in only considering irreducible
groups in Theorem~\ref{thm:main} and Corollary~\ref{cor:Coxeter}.
  This is meant in the following sense.
  Let~$W$ be a well-generated reducible complex reflection group and
write $W=W_1\times \dots \times W_k$, where the~$W_i$ are the
irreducible components.
  The Coxeter elements in~$W$ are, by definition, the product of the
Coxeter elements of each $W_i$, and the set of reflections in~$W$ is
the union of the set of reflections in all $W_i$, and one thus obtains
immediately that $\formula_W(t)$ can be computed by applying
Theorem~\ref{thm:main} to all its irreducible components,
  \[
    \formula_W(t) = \formula_{W_1}(t) \cdots \formula_{W_k}(t).
  \]
\end{remark}
\begin{remark}
Our proof of Theorem~\ref{thm:main} uses the classification of complex reflection groups. We first treat the case of the infinite families, and then check the finitely many remaining exceptional cases. This is possible since one can encode the a priori infinite sum $\formula_W(t)$ as a finite sum over all irreducible characters, see Formula~\eqref{eq:frobeniusZt} below. Of course, it is natural to hope for a uniform proof of this theorem, i.e., a proof that does not rely on the classification. However, we point out that even in the simpler case of Formula~\eqref{eq:DTZ}, whose proof goes back to~\cite{Deligne}, only a uniform recurrence formula is known that can be used to prove this formula case-by-case. Similarly, multiple results in~\cite{Chapoton,Armstrong,Reading,Bes2007,BR2007} rely on the classification either by directly using case-by-case analysis, or by relying on the counting formula for noncrossing partitions for which still no uniform proof is known either.
\end{remark}

In Section~\ref{sec:background}, we recall needed background on complex reflection groups. In Section~\ref{sec:dihedralgroup}, we prove the main theorem for the cyclic groups and the dihedral groups by explicit computations. In Section~\ref{sec:representationtheory}, we recall a general approach to the enumeration of factorizations in groups via representation theory. In Section~\ref{sec:classicaltypes}, we prove the main theorem for the infinite families of well-generated complex reflection groups. In Section~\ref{sec:exceptionaltypes}, we then present computer verifications for the exceptional well-generated complex reflection groups. Since we were not able to provide references for explicit descriptions of the irreducible
characters of the infinite families of irreducible well-generated complex reflection groups, we describe them explicitly in Appendix~\ref{app:A}.
In the final Appendix~\ref{app:exceptionaltypes}, we provide parts of the data that we have used in the computer verification of the main theorem for exceptional groups, namely the evaluations of the irreducible characters of exceptional well-generated complex reflection groups at reflections and at one conjugacy class of Coxeter elements.

\section{Background on complex reflection groups}
\label{sec:background}

In this section, we recall some background on complex reflection groups.
All these results can be found for example in~\cite{LT2009} to which we also refer for further details.
We moreover adapt some notation from~\cite{BR2007}.

\medskip

Let $V = \CC^n$ be a complex vector space of dimension~$n$.
A \defn{(complex) reflection} is a linear transformation of $V$ that has finite order and whose fix space $\fixspace(w) := \ker(\one - w)$ is a hyperplane in $V$.
Such a hyperplane is called a \defn{reflecting hyperplane}.
A \defn{complex} (or \defn{unitary}) \defn{reflection group} $W$ is a finite subgroup of $\GL(V)$ generated by reflections.
We denote the set of reflections in~$W$ by~$\classR$ and the set of corresponding reflecting hyperplanes by $\classR^*$.
The space $V$ and the action of $W$ on $V$ are called \defn{reflection} (or \defn{natural}) \defn{representation} of~$W$.
We say that $W$ is \defn{irreducible} if this representation is irreducible, i.e. if no proper subspace of $V$ is stable under the action of~$W$.
In this case, the dimension~$n$ of $V$ is called the \defn{rank} of~$W$.
Throughout this paper, we assume that $W$ is irreducible. 
We say that two complex reflection groups are \defn{isomorphic} if they are isomorphic as abstract groups and if their reflection representations are isomorphic.

G.C.~Shephard and J.A.~Todd classified complex reflection groups in~\cite{ST1954}.
They moreover used this classification to show that when $W$ acts on the symmetric algebra $S = \Sym(V^*) \cong \CC[x_1,\dots,x_n]$, its ring of invariants $S^W$ is again a polynomial algebra, generated by homogeneous polynomials $f_1,\dots,f_n$ of uniquely determined degrees $d_1 \leq \dots \leq d_n$. These are called the \defn{degrees} of $W$. This result was as well proved by C.~Chevalley in~\cite{Che1955}.

It turns out that there is an equivalent way to define the degrees using the \defn{coinvariant algebra} $S/{\langle \mathbf{f} \rangle}$, where $\langle \mathbf{f} \rangle = \langle f_1,\ldots,f_n \rangle = S_+^W$ is the ideal in $S$ generated by all invariants without constant term. Both Shephard and Todd \cite{ST1954} and Chevalley \cite{Che1955} showed that $S/{\langle \mathbf{f} \rangle}$ carries the regular representation of $W$.
Thus, $S/{\langle \mathbf{f} \rangle}$ contains exactly $k$ copies of any irreducible $W$-representation $U$ of dimension $k$. In particular $S/{\langle \mathbf{f} \rangle}$ contains~$n$ copies of $V$.
The multiset of \defn{$U$-exponents} $e_1(U),\ldots,e_k(U)$ is given by the degrees of the homogeneous components of $S/{\langle \mathbf{f} \rangle}$ in which these $k$ copies of $U$ occur.
It is known, see e.g.~\cite[Section~4.1]{LT2009}, that the degrees of $W$ are uniquely determined by saying that the $V$-exponents are equal to $d_1-1,\ldots,d_n-1$.
This characterization has the advantage that one can as well define the \defn{codegrees} $d_1^* \geq \ldots \geq d_n^*$ by saying that the $V^*$-exponents are given by $d_1^*+1,\ldots,d_n^*+1$.
It is well known that the degrees and the codegrees determine the number of reflections and the number of reflecting hyperplanes in $W$,
\begin{align}
  |\classR| = (d_1-1) + \ldots + (d_n-1) \qquad
  |\classR^*| = (d^*_1+1) + \ldots + (d^*_n+1), \label{eq:reflectionshyperplanes}
\end{align}
see~\cite[Theorem~4.14(ii) and Appendix~C,~Section~1.2]{LT2009}.
As in~\cite[Remark~6.13]{Mal99}, we say that $W$ is \defn{well-generated} if the following two equivalent properties hold,
\begin{enumerate}[(i)]
  \item $W$ is generated by~$n$ reflections,
  \item the degrees and the codegrees satisfy $d_i + d_i^* = d_n$. \label{eq:degreecodegree}
\end{enumerate}
The equivalence of these two properties was observed in~\cite{OS1980}, using the Shephard-Todd classification.
For well-generated complex reflection groups, the \defn{Coxeter number} $h$ is defined to be the largest degree $d_n$ of $W$.
Observe that \eqref{eq:degreecodegree} together with the counting formulas in \eqref{eq:reflectionshyperplanes} yields
\begin{align}
  |R|+|R^*| = nh. \label{eq:nhRR}
\end{align}
An element $c$ in a complex reflection group $W$ is called \defn{regular} if it has an eigenvector lying in the complement of the reflecting hyperplanes for $W$ and furthermore $\zeta$-\defn{regular} if this eigenvector may be taken to have eigenvalue~$\zeta$.
In this case, the multiplicative order $d$ of $\zeta$ is called a \defn{regular number} for~$W$.
An integer~$d$ is a regular number if and only if it divides as many degrees as codegrees, see e.g.~\cite[Theorem~11.28]{LT2009}.
If~$W$ is well-generated, $d_n = h$ and $d^*_n = 0$ are the only degrees and codegrees that are integer multiples of~$h$.
Thus,~$h$ is always a regular number in this situation.
For any primitive $h$-th root of unity $\zeta$ there thus exists a regular element $c_\zeta \in W$ with eigenvalue $\zeta$, see~\cite[Remark~11.23]{LT2009}.
Any such element is called \defn{Coxeter element}.
For any $\zeta$, all $\zeta$-regular elements are $W$-conjugate, see~\cite[Corollary~11.25]{LT2009}.
Thus, the class $C_\zeta$ of Coxeter elements for a particular primitive $h$-th root of unity $\zeta$ is closed under conjugation.
Nonetheless, observe that according to our definition, there are in general more than a single conjugacy class of Coxeter elements, since the conjugacy classes for the various primitive $h$-th roots of unity may differ.
For comparison, we notice that the paper~\cite{Bes2007} defines Coxeter elements as $e^{\frac{2i\pi}{h}}$-regular elements, i.e. elements of the class $C_{\zeta}$ with $\zeta=e^{\frac{2i\pi}{h}}$. Therefore all Coxeter elements in the sense of \cite{Bes2007} are also Coxeter elements with our definition, but the converse is not true in general.

To check that an element of $W$ is a Coxeter element, we will use several times the following necessary and sufficient condition.
\begin{proposition}
  An element $w \in W$ is a Coxeter element if and only if it has an eigenvalue that is a primitive $h$-th root of unity.
\end{proposition}
This result is well known to specialists, but we have not found it explicitly in the literature in the generality discussed here.
For real groups, a proof can be found for example in~\cite[Theorem~32-2C]{Kan2001}. For completeness, we provide a proof of the general case that follows exactly the same lines. The arguments are very similar to the ones used in the proof of~\cite[Theorem~11.15]{LT2009}.
\begin{proof}
  Since $W$ is irreducible and well-generated, we have
  \begin{align}
    h = d_n > d_{n-1},\ldots,d_2,d_1. \label{eq:hismax}
  \end{align}
  Fix now $\zeta$ to be a primitive $h$-th root of unity. We have already seen that there exists a regular element $c_\zeta \in W$ with eigenvalue $\zeta$, and that its conjugacy class $\classCox_{\zeta}$ is the conjugacy class of all $\zeta$-regular elements in $W$.
  It follows from~\cite[Theorem~11.24(iii)]{LT2009} and~\eqref{eq:hismax} that
  $$|\classCox_\zeta| = |W|/|Z(c_\zeta)| = d_1 \cdots d_{n-1},$$
  where $Z(c_\zeta)$ denotes the centralizer of $c_\zeta$, and where we used the fact that $|W| = d_1 \cdots d_n$.
  On the other hand, we have that the Pianzola-Weiss polynomial for $\zeta$ is
  $$P(T) := \sum_{w \in W} T^{\dim{V(w,\zeta)}} = d_1 \cdots d_{n-1}(T+h-1),$$
  where $V(w,\zeta) = \{ v \in V : w(v) = \zeta \cdot v\}$ is the eigenspace for $w$ with eigenvalue~$\zeta$, see~\cite[Corollary~10.39]{LT2009}.
  The coefficient $d_1 \cdots d_{n-1}$ of $T$ in $P(T)$ counts the number of elements in $W$ that have an eigenvalue $\zeta$.
  Since all elements in $\classCox_\zeta$ have $\zeta$ as an eigenvalue, we conclude that all elements in $W$ that have $\zeta$ as an eigenvalue are already contained in $\classCox_\zeta$, and are thus Coxeter elements. The statement follows.
\end{proof}

\subsection{The classification of well-generated complex reflection groups}
\label{sec:classification}

\newcommand{\Goon}{\widehat G(1,1,n)}
In the remainder of this section, we recall the classification of irreducible well-generated complex reflection groups from~\cite{ST1954}. A \emph{monomial} matrix is a square matrix with exactly one non-zero entry in each row and column. Let $G(r,p,n)$ with $p$ dividing~$r$ be the group of all monomial $(n \times n)$-matrices with entries being~$r$-th roots of unity, such that the product of the non-zero entries is an $(r/p)$-th root of unity. It is a subgroup of $GL(\CC^n)$, and it has order $n! r^n/p$.
Notice that the special case $G(1,1,n)$ is the group of all permutation matrices of size~$n$. This group leaves invariant the proper subspace $V_0$ of $\CC^n$ formed by vectors whose coordinates sum to zero. We denote by $\Goon$ the restriction of $G(1,1,n)$ to this subspace. Hence $\Goon \subset GL(V_0)$ is an irreducible complex reflection group of rank $n-1$.

\begin{theorem}[Shephard, Todd]\label{thm:shephardtodd}
Let $W$ be an irreducible complex reflection group. 
Then $W$ is isomorphic to one of the following complex reflection groups:
\begin{itemize}
\item the group $G(r,p,n)$ for some integers $r\geq 2$ and $n,p \geq 1$ such that $p$ divides~$r$ (if $n=1$, we impose moreover that $p=1$).
\item the symmetric group $\Goon$ for some $n\geq 5$;
\item one of $34$ exceptional groups.
\end{itemize}
\end{theorem}
Following \cite{ST1954}, we denote the $34$ exceptional \defn{Shephard-Todd classification types} by $G_4$ through $G_{37}$; the degrees and codegrees of all irreducible complex reflection groups were already computed therein.
For the groups $G(r,p,n)$, the degrees are given by
\begin{align}
  r,2r,\dots,(n-1)r,rn/p, \label{eq:degrees}
\end{align}
while the codegrees are given by
\begin{align}
  0,r,\dots,(n-1)r &\text{ if } p < r, \nonumber \\
  0,r,\dots,(n-2)r,(n-1)r-n &\text{ if } p = r. \label{eq:codegrees}
\end{align}
For the symmetric group $\Sn_n\simeq\Goon$, the degrees and codegrees are given by $2,3,\dots,n$ and $0,1,\dots,n-2$, respectively.
\begin{table}[t]
  \centering
  \begin{tabular}{c || l}
    rank & Well-generated exceptional classification types \\
    \hline
    \hline
    &\\[-10pt]
    $2$ & $G_4,G_5,G_6,G_8,G_9,G_{10},G_{14},G_{16},G_{17},G_{18},G_{20},G_{21}$ \\
    \hline
    &\\[-10pt]
    $3$ & $G_{23}=H_3,G_{24},G_{25},G_{26},G_{27}$ \\
    \hline
    &\\[-10pt]
    $4$ & $G_{28}=F_4,G_{29},G_{30}=H_4,G_{32}$ \\
    \hline
    &\\[-10pt]
    $5$ & $G_{33}$ \\
    \hline
    &\\[-10pt]
    $6$ & $G_{34},G_{35}=E_6$ \\
    \hline
    &\\[-10pt]
    $7$ & $G_{36}=E_7$ \\
    \hline
    &\\[-10pt]
    $8$ & $G_{37}=E_8$ \\
  \end{tabular}
  \caption{All well-generated exceptional classification types.}
  \label{tab:exceptionaltypes}
\end{table}
This implies that the three infinite families of irreducible well-generated complex reflection groups are given by $\Goon$, $G(r,1,n)$ and $G(r,r,n)$. Observe that the infinite families of real reflection groups are the particular cases
\begin{itemize}
  \item $G(r,r,2) = I_2(r)$,
  \item $\Goon = A_{n-1}$,
  \item $G(2,1,n) = B_n$,
  \item $G(2,2,n) = D_n$.
\end{itemize}
Observe also that the group $G(r,1,1)$ is the cyclic group $C_r$ of order~$r$.
The complete list of well-generated exceptional Shephard-Todd classification types is given in Table~\ref{tab:exceptionaltypes}, their degrees and codegrees can be found in Table~\ref{tab:exceptionaldegreescodegrees} in Appendix~\ref{app:exceptionaltypes}.
We will later use this classification to prove Theorem~\ref{thm:main}.
 In particular, see Appendix~\ref{app:exceptionaltypes} for the representation theoretic data for the exceptional types.

\section{Direct proof for the cyclic groups and dihedral groups}
\label{sec:dihedralgroup}

As a warmup exercise, we prove the main theorem for the cyclic group $C_r$ and for the dihedral group $I_2(r)$. Observe that the cyclic groups will as well be covered in Section~\ref{par:Gr1n}, while the proof for the dihedral groups will not be covered in Section~\ref{par:Grrn} and is thus only given here.

\smallskip

\noindent{\bf The proof for the cyclic group $G(r,1,1)=C_r$ with $r\geq 2$.} The cyclic group $G(r,1,1)=C_r$ formed by the~$r$-th roots of unity is a complex reflection group of rank $n=1$. It has one reflecting hyperplane, and since any element of the group different from $1$ is a reflection, it contains $(r-1)$ reflections. Let $\zeta$ be a primitive~$r$-th root of unity. Since the Coxeter number is $h=r$, the element $\zeta$ is a Coxeter element.
For $\ell \geq0$, let $f_\ell$ be the number of factorizations of $\zeta$ into a product of $\ell$ reflections. Then $(r-1)^\ell-f_{\ell}$ is the number of $\ell$-tuples of reflections $(\zeta_1,\zeta_2,\dots,\zeta_\ell)$ whose product is not equal to $\zeta$. Given such a tuple, there is a unique reflection $\zeta_{\ell+1}$ such that $\zeta_1\zeta_2\dots\zeta_\ell\zeta_{\ell+1} = \zeta$.  We thus obtain the recurrence relation
$$
f_{\ell+1} = (r-1)^\ell - f_\ell.
$$
Given that $f_0=0$ and $f_1=1$, this implies that $f_\ell = \frac{1}{r}\big((r-1)^\ell-(-1)^\ell\big)$ for all $\ell \geq 0$. We thereby obtain
$$
\formula_{C_r}(t) = \sum_{\ell\geq 0} \frac{t^\ell}{\ell!}f_\ell = \frac{1}{r}\left(e^{(r-1)t} - e^{-t}\right),
$$
which yields Formula~\eqref{eq:coxetergenus} in this case.

\smallskip

\noindent{\bf The proof for the dihedral group $G(r,r,2)=I_2(r)$ with $r\geq 2$.} Let $\zeta$ be a primitive~$r$-th root of unity. The dihedral group $G(r,r,2)=I_2(r)$ consists of the~$r$ reflections $s_i = \left(\begin{smallmatrix} 0&\zeta^i\\ \zeta^{-i}&0 \end{smallmatrix}\right)$ and the~$r$ elements $t_i = \left(\begin{smallmatrix} \zeta^i&0\\ 0&\zeta^{-i} \end{smallmatrix}\right)$, for $1 \leq i \leq r$ (the elements $t_i$ are the ``rotations'' in the classical representation of $I_2(r)$ as the symmetry group of the regular~$r$-gon).
Notice that the hyperplanes $\ker(\one-s_i)$ are all distinct, so that there are~$r$ reflecting hyperplanes. The Coxeter number is $h=r$, which shows that $c=s_1 s_0 = t_1$ is a Coxeter element.
Note that $s_i s_j = c$ if and only if $i = j+1\ (\mod r)$.
In particular, since for any integer $g \geq 0$ any product of $2g+1$ reflections is again a reflection, there is a unique other reflection such that the complete product of all $2g+2$ reflections equals $c$.
Therefore the number of factorizations of $c$ into $2g+2$ reflections is equal to $r^{2g+1}$. We thereby obtain
$$\formula_{I_2(r)}(t) = \sum_{g\geq0}\frac{t^{2g+2}}{(2g+2)!}r^{2g+1} 
= \frac{e^{rt}+e^{-rt}-2}{2r}=
\frac{1}{2r}\Big(e^{rt/2}-e^{-rt/2}\Big)^2,$$
which yields Formula~\eqref{eq:coxetergenus} in this case.

\section{A representation theoretic approach to the main theorem}
\label{sec:representationtheory}

In this section, we recall a classical approach to the enumeration of
factorizations in groups via representation theory. This approach is widely used
in the literature to enumerate factorizations in the symmetric group, see
e.g.~\cite[Appendix~A]{LandoZvonkine} and the references therein. It is
based on the Frobenius formula given in Formula~\eqref{eq:frobenius} below.
For completeness, we recall its proof here, following the same lines as in~\cite[Section~A.1.3]{LandoZvonkine}.

Let $W$ be an irreducible well-generated complex reflection group.
The group algebra $\CC[W]$ is decomposed into irreducible $W$-modules as
\begin{align}\label{eq:regular}
 \CC[W] = \bigoplus_{\lambda \in \Lambda} \dim(\lambda) V^\lambda,
\end{align}
where $V^\lambda$, $\lambda \in \Lambda$ is a complete list of irreducible
representations of $W$.
We let $\classR\subset W$ be the union of the conjugacy classes of all reflections. We use the same letter, in non-calligraphic style, to denote the corresponding sum in the group algebra,
$$
 R := \sum_{\tau \in \classR} \tau.
$$
Moreover, we extend class functions linearly from $W$ to $\CC[W]$. In particular if $\chi :W\rightarrow \CC$ is a class function, we use the notation
$$
\chi(R) := \sum_{\tau \in \classR} \chi(\tau).
$$

We consider the action of $W$ on $\CC[W]$ by left multiplication.
Since an element $w \in W$ acts on $\mathbb{C}[W]$ by permuting the canonical basis, the trace of this action is equal to the number of fixed points under this action, which is $|W|$ if $w=1$ and $0$ otherwise.
Since $R$ is in the center of $\CC[W]$, it acts as a scalar on every irreducible representation, so the decomposition~\eqref{eq:regular} gives the Frobenius formula
\begin{align}
\#\Big\{ (\tau_1,\tau_2,\dots,\tau_\ell)\in \classR^\ell \ : \ \tau_1 \tau_2 \dots &\tau_\ell = c \Big\} \nonumber \\
&= \frac{1}{|W|} \trace_{\mathbb{C}[W]} \Big( {R}^{\ell} c^{-1} \Big) \nonumber \\
&= \frac{1}{|W|} \sum_{\lambda \in \Lambda} \dim(\lambda) \trace_{V^\lambda} \Big( {R}^{\ell} c^{-1} \Big) \nonumber\\
&= \frac{1}{|W|} \sum_{\lambda \in \Lambda} \dim(\lambda)^{1-\ell} \chi_\lambda({R})^{\ell} \chi_\lambda(c^{-1})\label{eq:frobenius},
\end{align}
where $\chi_\lambda$ is the character of the representation $\lambda$.
Using~\eqref{eq:frobenius}, we can now deduce that $\formula_W(t)$ is given by the finite sum
\begin{align}
\label{eq:frobeniusZt}
  \formula_W(t) &=
    \frac{1}{|W|} 
    \sum_{\lambda \in \Lambda} \dim(\lambda)
    \chi_\lambda(c^{-1})
    \exp\left(t \cdot \frac{\chi_\lambda({R})}{\dim(\lambda)}\right).
\end{align}

We will use this formula, together with the classification of complex reflection groups and the
knowledge of their irreducible representations to prove Theorem~\ref{thm:main}.
In order to simplify notation, we will often make use of the \defn{normalized character} defined as the character divided by the dimension of the corresponding representation $\chi_\lambda(\one) = \dim(\lambda)$.
By convention, a normalized character will be denoted with a tilde, such as
$$\widetilde\chi_\lambda := \frac{1}{\dim(\lambda)} \cdot \chi_\lambda.$$

\section{Proof of the main theorem for the infinite families}
\label{sec:classicaltypes}

In this section, we treat the case of the infinite families of
well-generated complex reflection groups, namely $\Goon\simeq \Sn_n$, $G(r,1,n)$ and $G(r,r,n)$.
For technical reasons we exclude the case $G(r,r,2)$, which has been treated in Section~\ref{sec:dihedralgroup}.
As mentioned in the introduction the case of  $\Goon\simeq \Sn_n$ is already known, but we discuss it here to set up notations and to recall some needed results. Our approach for this case is similar e.g. to~\cite{SSV} or to~\cite[Appendix~A.2.4]{LandoZvonkine}.

\subsection{The proof for the symmetric group \texorpdfstring{$\Goon\simeq\Sn_n$}{Sn}}

The symmetric group $\Sn_n$ is the group of all permutations of $\{1,2,\dots,n\}$. 
Its irreducible representations are classically indexed by the set of
partitions of~$n$, see e.g.~\cite{Sagan}.
In this section we are going to use extensively two classical results, the
\defn{hook-length formula} and the \defn{Murnaghan-Nakayama rule}.
These formulas give the dimensions of these representations and the evaluation of
their characters, respectively.
We refer the reader to the original papers~\cite{FRT,Mur37,Nak40} or to~\cite[Sections~3.10 and~4.10]{Sagan} for a detailed treatment of these results.
We will use them several times without again recalling these references.

A \defn{partition} of~$n$ is a non-increasing sequence $\lambda=[\lambda_1,\lambda_2,\dots,\lambda_r]$ of positive integers summing to~$n$.
The integer~$n$ is called the \defn{size} of $\lambda$, denoted by $|\lambda|=n$, and we write $\lambda \vdash n$.
The integers $\lambda_i$ are called the \defn{parts} of $\lambda$.
To denote a partition, we often use superscript notation, we write for example $\lambda=[4^2, 3^1, 1^3]$ for the partition $\lambda = [4,4,3,1,1,1]$ of $14$.
By convention, the \defn{empty partition} $\varnothing$ is the unique partition of~$0$.

If not otherwise stated, we denote by $\chi_\lambda$ the character of the representation of $\Sn_n$ indexed by the partition $\lambda \vdash n$.

\medskip

We choose the long cycle $c=(1,2,\dots,n)$ as a Coxeter element. Since all Coxeter elements are conjugated to $c$ in $\Sn_n$, it is enough to prove Theorem~\ref{thm:main} for this particular element.
Since $c$ is a long cycle, as is its inverse $c^{-1}$, the Murnaghan-Nakayama rule implies that $\chi_\lambda(c^{-1})=0$ unless $\lambda$ is a \defn{hook}, i.e., unless $\lambda$ is of the form $\hook{n}{k} := [n-k,1^{k}]$ for some $0\leq k < n$.
The following lemma is a straightforward consequence of the hook-length formula and of the Murnaghan-Nakayama rule, and will be useful in several places.

\begin{lemma}[\bf Hook characters: dimension and evaluation on a long cycle]
\label{lemma:hooks}
Let $n \geq 1$ and $0\leq k < n$. Then
$$
\dim(\hook{n}{k}) = \binom{n-1}{k} \quad \text{and} \quad \chi_{\hook{n}{k}} \big((1,2,\dots,n)\big) =(-1)^k.
$$
\end{lemma}
The reflections of the group $\Goon \simeq \Sn_n$ are given by the transpositions $(i,j)$ for $1\leq i < j \leq n$.
Hence in view of evaluating~\eqref{eq:frobeniusZt}, we will also need the value of hook characters on transpositions.
If $\lambda=[\lambda_1,\lambda_2,\dots,\lambda_r] \vdash n$, we define the multiset of its contents as
$$
\content(\lambda) :=  \bigmultiset{ j-i : 1 \leq i \leq r, 1 \leq j \leq \lambda_i },
$$
where we indicate the multiset notation by double braces. Notice that $\content(\lambda)$  has cardinality~$n$.
We have the following classical lemma relating character evaluations and contents.
\begin{lemma}[\bf Normalized character evaluated on a transposition]
\label{lemma:contents}
Let $\lambda \vdash n$ be a partition and $\tau \in \Sn_n$ be a transposition.
Then we have
$$
\widetilde\chi_\lambda(\tau) = \frac{2}{n(n-1)} \sum_{ x\in
\content(\lambda)} x .
$$
\end{lemma}
\begin{proof}
As in Section~\ref{sec:representationtheory}, we let $R=\sum_{\tau\in\classR}\tau$ be the sum of all transpositions in the group algebra $\CC[\Sn_n]$.
Write $R=\sum_{i=1}^{n} J_i$ where $J_i = \sum_{i<j} (i,j)$ is the $j$-th Jucys-Murphy element. 
It is well known --~see e.g. \cite{OkounkovVershik}~-- that given any symmetric function $f$ in~$n$ variables, the element $f(J_1,J_2,\dots,J_n)$ acts on the irreducible module $V^\lambda \subset \CC[\Sn_n]$ as the scalar $f(\content(\lambda))$.
The statement follows by taking $f(x_1,x_2,\dots x_n)=x_1+x_2+\dots +x_n$, and observing that $\chi_\lambda(R) = \frac{n(n-1)}{2}\chi_\lambda(\tau)$.
\end{proof}
In the particular case of hooks, we obtain
\begin{align}
\widetilde\chi_{\hook{n}{k}}(\tau) &=  \frac{n-2k-1}{n-1}. \label{eq:hooktranspo}
\end{align}
We now recalled all ingredients needed to provide a proof of Theorem~\ref{thm:main} for the symmetric group.
\begin{proof}[Proof of Theorem~\ref{thm:main} for the symmetric group $W = \Goon$]
Using Equation~\eqref{eq:hooktranspo} together with Lemma~\ref{lemma:hooks}, Formula~\eqref{eq:frobeniusZt} can be rewritten as
\begin{align*}
\formula_W(t) &= \frac{1}{|W|} 
\sum_{k =0}^{n-1} \binom{n-1}{k}
  (-1)^k e^{t\frac{n(n-2k-1)}{2}}\\
&= \frac{1}{|W|} e^{t\binom{n}{2}}
 \left(1-e^{-tn}\right)^{n-1}
\\
&=
\frac{1}{|W|}\left(e^{t\frac{n}{2}}-e^{-t\frac{n}{2}}\right)^{n-1}.
\end{align*}
Since $\Goon$ has rank $n-1$, and $|\mathcal{R}|=|\mathcal{R}^*|=\binom{n}{2}$, this coincides with~\eqref{eq:coxetergenus} in this case.
\end{proof}

\subsection{The structure of the proof}

\label{par:4step}
In the cases of $G(r,1,n)$ and $G(r,r,n)$, the structure of the proof will be similar to what we just did for the symmetric group.
It can be decomposed into the following four steps.
\begin{enumerate}[\bf step 1]
\item Provide a list of all irreducible characters not vanishing on the inverse of the Coxeter element (for $\Sn_n$, these were exactly the hook characters).
\item Compute the dimension of these characters and their evaluation on the inverse of the Coxeter element (for $\Sn_n$, this was given by Lemma~\ref{lemma:hooks}).
\item Compute the evaluation of these characters on the conjugacy classes of reflections (for $\Sn_n$, this was given by Equation~\eqref{eq:hooktranspo}).
\item Use Formula~\eqref{eq:frobeniusZt} and simplify the sum (for $\Sn_n$, we just applied Newton's binomial formula) to prove Formula~\eqref{eq:coxetergenus}.
\end{enumerate}

\begin{remark}
In Section~\ref{par:Grrn}, we prove the main theorem for the group
$G(r,r,n)$ following the four steps above. In order to do that, we 
list the irreducible characters of $G(r,r,n)$ that do not vanish on
 a Coxeter element (Lemma~\ref{lemma:step1Grr}). 
Note that since $G(r,r,n)$
is a subgroup of $G(r,1,n)$, 
another approach would be to apply
the Frobenius formula directly in the group $G(r,1,n)$%
\footnote{The only property we need for that is that any conjugate in $G(r,1,n)$ of a reflection in $G(r,r,n)$ is a
reflection in $G(r,r,n)$, which is true. Note that a Coxeter element $c$ of $G(r,r,n)$ is \emph{not} a Coxeter element of the group $G(r,1,n)$, but this does not prevent one to apply the Frobenius formula.}. This other approach would save us the description of the irreducible characters of $G(r,r,n)$ given in Appendix~\ref{sec:irreduciblecharactersrrn}.
We follow the first approach here because we think that Lemma~\ref{lemma:step1Grr} is interesting in itself. Indeed, together with the results for $G(r,1,n)$ and with the data for irreducible groups given in Appendix~\ref{app:exceptionaltypes}, this provides an exhaustive list of irreducible characters non vanishing on a Coxeter element in well-generated complex reflection groups, that may be useful in other situations.

\end{remark}

\subsection{The proof for the group \texorpdfstring{$G(r,1,n)$}{G(r,1,n)} with \texorpdfstring{$r\geq 2$}{r>=2}}
\label{par:Gr1n}

The Coxeter number for $G(r,1,n)$ is $h = nr$. Let $\zeta$ be a primitive $h$-th root of
unity, and let $\xi = \zeta^n$, which is a primitive~$r$-th root of unity.
Recall that $G(r,1,n)$ is the group of all monomial matrices whose non-zero entries are powers of $\xi$.
Given $\sigma\in\Sn_n$ and $(i_1,i_2,\dots,i_n)$ with $0 \leq i_1,\ldots,i_n < r$, we denote by
$w = \sigma \wprod (i_1,i_2,\dots,i_n)$ the element of $G(r,1,n)$ with entry $\xi^{i_\ell}$ at position $(\ell,\sigma_\ell)$ for $1\leq \ell \leq n$.
Observe that our presentation depends on the choice of $\xi$ and thus on the choice of $\zeta$.
We moreover denote by $|w| := \sigma$ the projection onto $\Sn_n$, and by $\|w\| := i_1+\ldots+i_n\ (\mod r)$.

Consider the element $c=c_0\wprod (0,0,\dots,0,1) \in G(r,1,n)$ where $c_0 \in \Sn_n$ is given by the long cycle $(1,2,\dots,n)$. We claim that $c$ is a Coxeter element of $G(r,1,n)$.
To see that, according to the discussion preceding Section~\ref{sec:classification}, it is enough to check that $c$ has an eigenvalue equal to $\zeta$. This is indeed the case, since the fact that $\xi = \zeta^n$ implies that the element $c$ has~$n$ eigenvalues given by $\big\{ \zeta, \zeta^{r+1},\zeta^{2r+1},\ldots,\zeta^{(n-1)r+1} \big\}$.
\begin{remark}\label{rem:coxeterClasses}
In the rest of this section, we will prove Theorem~\ref{thm:main} for the group $G(r,1,n)$, for the particular choice of Coxeter element $c=c_0\wprod (0,0,\dots,0,1)$. Since $c$ belongs to the conjugacy class $\classCox_{\zeta}$ of all Coxeter elements which are $\zeta$-regular, and since the number of factorizations into reflections is the same for elements in the same conjugacy class, this will also prove the theorem for all elements $c'\in\classCox_\zeta$. Finally, since we fixed $\zeta$ as an \emph{arbitrary} primitive $h$-th root of unity, this will imply the result for \emph{all} Coxeter elements of the group $G(r,1,n)$.
\end{remark}

The following proposition is well known and describes the irreducible characters of $G(r,1,n)$.
We refer to Appendix~\ref{sec:irreduciblecharactersr1n} for its proof and references.
\begin{proposition}\label{prop:Gr1n}
The complete and unambiguous list of irreducible characters of $G(r,1,n)$ is
obtained as follows. Let $\vec\lambda =
(\lambda^{(0)},\lambda^{(1)},\dots,\lambda^{(r-1)})$ be an~$r$-tuple of
partitions of total size~$n$, and let $k_\ell=|\lambda^{(\ell)}|$. One may view
$$
B:=G(r,1,k_0)\times G(r,1,k_1)\times \dots \times G(r,1,k_{r-1})
$$
as a subgroup of $G(r,1,n)$, formed by block-diagonal matrices.
The character $\chi_{\vec\lambda}$ of $G(r,1,n)$ is then given by
\begin{align}\label{eq:charGr1n}
\chi_{\vec\lambda}(w)
&= \frac{1}{|B|}
  \sum_{\substack{s \in G(r,1,n) \\ s^{-1}ws \in B}}
\prod_{\ell=0}^{r-1} \chi_{\lambda_\ell}(|w_\ell|)\cdot  \xi^{\ell\cdot \|w_\ell\| },
\end{align}
where $s^{-1}ws \in B$ in the sum is denoted by $(w_0,\ldots,w_{r-1})$, and where $\chi_{\lambda^{(\ell)}}$ denotes the $\Sn_{k_\ell}$-character indexed by
$\lambda^{(\ell)}$.
\end{proposition}
Having the list of irreducible characters at hand, we can now proceed with the four steps in the proof as explained in Section~\ref{par:4step}.

\step{1} The following lemma describes those characters not vanishing on the inverse of the Coxeter element. 
\begin{lemma}
  The character $\chi_{\vec\lambda}$ defined in~\eqref{prop:Gr1n} vanishes on $c^{-1}$ unless
  $$ \vec\lambda = \ghook{n}{k}{q} := (0,\dots,0,\hook{n}{k},0,\dots0) \text{ for } 0 \leq q < r \text{ and } 0 \leq k < n, $$
  where as before $\hook{n}{k}$ is a hook of size~$n$, and where the hook appears at position~$q$.
\end{lemma}
\begin{proof}
  Let us consider the evaluation of Formula~\eqref{eq:charGr1n} with $w=c^{-1}$. Observe that $|c^{-1}|$ is a long cycle. Thus, the cyclic group it generates acts transitively on  $\{1,2,\dots,n\}$, which implies that the sum in \eqref{eq:charGr1n} is empty unless exactly one of the numbers $k_0,\ldots,k_{r-1}$ is equal to~$n$, say $k_q$.  Moreover, the Murnaghan-Nakayama rule implies that the evaluation $\chi_{\lambda_q}(c_0)$ is zero unless $\lambda^{(k_q)}$ is a hook. We are therefore left with all the characters listed in the lemma.
\end{proof}

\step{2} By the Murnaghan-Nakayama rule and the fact that $\|c^{-1}\|=r-1$, the evaluation of~\eqref{eq:charGr1n} gives
$$
\chi_{\ghook{n}{k}{q}}(c^{-1}) = (-1)^k \xi^{-q}.
$$
Moreover, their degrees are easily determined by evaluating~\eqref{eq:charGr1n} on the identity element
$$
\dim(\ghook{n}{k}{q}) = \chi_{\hook{n}{k}}(1) = \binom{n-1}{k} \mbox{ by Lemma~\ref{lemma:hooks}}.
$$

\step{3} The set of reflections of $G(r,1,n)$ is divided into~$r$ conjugacy classes, 
$\classR = \bigcup_{\ell=0}^{r-1} \classR_\ell$, where
\begin{itemize}
\item for $1 \leq \ell <r$, $\classR_\ell$ is the set of matrices obtained from the
identity matrix by replacing one of its entries by $\xi^\ell$. Hence $\classR_\ell$ has
cardinality~$n$, for each~$\ell$.
\item $\classR_0$ is the set of matrices $M_{i,j}^{(k)}$, for $1\leq i<j\leq n$, $1\leq k \leq r$, where $M_{i,j}^{(k)}$ is obtained from the matrix of the transposition $(i,j)$ by replacing its entries in columns $i$ and $j$ by $\xi^k$ and $\xi^{-k}$, respectively.
Hence, $\classR_0$ has cardinality~$r\binom{n}{2}$.
\end{itemize}
For $0\leq \ell < r$ we let $R_\ell = \sum_{\tau \in \classR_\ell} \tau$ be the
corresponding element of the group algebra.
We need to evaluate the character $\chi_{\ghook{n}{k}{q}}$ on each of the
elements $R_\ell$.
 If $\tau \in \classR_\ell$ with $\ell > 0$, then $\|\tau\| = \ell$ and $|\tau| = \one$.
Therefore, evaluating~\eqref{eq:charGr1n} yields $\chi_{\ghook{n}{k}{q}}(\tau) = \dim(\hook{n}{k})\xi^{q\ell}$, and thus
$$
\chi_{\ghook{n}{k}{q}}(R_\ell) = 
n \dim(\hook{n}{k}) \xi^{q\ell} = 
n \dim(\ghook{n}{k}{q}) \xi^{q\ell}.
$$
If $\tau \in \classR_0$, then $\|\tau\| = 0$ and $|\tau|$ is a transposition, so by~\eqref{eq:hooktranspo} the character~\eqref{eq:charGr1n} evaluates as
$$
  \chi_{\ghook{n}{k}{q}}(\tau)
    = \chi_{\hook{n}{k}} (|\tau|)
    = \dim(\hook{n}{k}) \frac{n-2k-1}{n-1}
    = \dim(\ghook{n}{k}{q}) \frac{n-2k-1}{n-1}.
$$
Adding all contributions and letting $R$ be the sum of all reflections as in
Section~\ref{sec:representationtheory}, we obtain the evaluation of the normalized character
$$
\widetilde\chi_{\ghook{n}{k}{q}} (R) = \sum_{\ell=1}^{r-1} n \xi^{q\ell} + \frac{n r (n-2k-1)}{2}.
$$

\step{4} From Steps~1--3, Formula~\eqref{eq:frobeniusZt} for $W=G(r,1,n)$ finally rewrites as
\begin{align*}
\formula_W(t)
&= \frac{1}{|W|} 
\sum_{q=0}^{r-1}\sum_{k=0}^{n-1}
\binom{n-1}{k} (-1)^k \xi^{-q}
\exp\left(t \sum_{\ell=1}^{r-1} n \xi^{q\ell}
 +t r\frac{ n(n-2k-1)}{2}
\right)
\\
&= \frac{1}{|W|}
e^{t r\binom{n}{2}}
\left(
\sum_{q=0}^{r-1} \xi^{-q}
e^{ t \cdot \sum_{l=1}^{r-1} n \xi^{ql}}
 \right)
\left(\sum_{k=0}^{n-1}
\binom{n-1}{k}(-1)^k
e^{-t rnk}
\right)
\\
&= \frac{1}{|W|}
e^{t r\binom{n}{2}
}
\left(1-e^{-t rn}\right)^{n-1}
\left(
\sum_{q=0}^{r-1} \xi^{-q}
e^{ t \cdot \sum_{\ell=1}^{r-1} n \xi^{q\ell}}
 \right).
\end{align*}
To finish the proof of Theorem~\ref{thm:main} for the group $G(r,1,n)$, we use the following lemma.
\begin{lemma}
\label{lem:zetasum}
  Let $\xi$ be a primitive~$r$-th root of unity. Then
  $$
  \sum_{q=0}^{r-1} \xi^{-q}
  e^{ t \cdot \sum_{\ell=1}^{r-1} n \xi^{q\ell}}
  =
  e^{t(r-1)n}
  -e^{-tn}
  $$
\end{lemma}
\begin{proof}
  If $1<q<r-1$ then $\xi^{-q}\neq 1$ and $\sum_{\ell=1}^{r-1}\xi^{q\ell} = -1$
  by summing the geometric progression (notice that here we use that $r\geq 2$). 
  We therefore obtain
  $$
    \sum_{q=0}^{r-1} \xi^{-q}
    e^{ t \cdot \sum_{\ell=1}^{r-1} n \xi^{q\ell}}
    =  \xi^0 e^{n(r-1)t} + \sum_{q=1}^{r-1} \xi^{-q} e^{-nt}
    = e^{(n-1)rt} - e^{-nt}.
  $$
\end{proof}
\begin{proof}[Proof of Theorem~\ref{thm:main} for the group $W = G(r,1,n)$ with $r\geq 2$]
  Using Lemma~\ref{lem:zetasum}, our last expression of $\formula_W(t)$ rewrites
  as
  \begin{align*}
    \formula_W(t)
    &= \frac{1}{|W|}
    e^{t r\binom{n}{2}
    }
    \left(1-e^{-t rn}\right)^{n-1}
    \left(e^{t(r-1)n}
    -e^{-t n}
    \right)\\
    &= \frac{1}{|W|}
    e^{t r\binom{n}{2} + t(r-1)n
    }
    \left(1-e^{-t rn}\right)^{n} \\
    &=
    \frac{1}{|W|}\left(e^{t\left(r\frac{n+1}{2}-1\right)}-e^{-t\left(r\frac{n-1}{2}+1\right)}\right)^{n}.
  \end{align*}
  Since we have for $W = G(r,1,n)$ that $|\classR| = r\binom{n+1}{2}-n$ and $|\classR^*| = r\binom{n}{2}+n$, this coincides with~\eqref{eq:coxetergenus} in this case.
\end{proof}

\subsection{The proof for the group \texorpdfstring{$G(r,r,n)$}{G(r,r,n)} with \texorpdfstring{$r\geq 2$}{r>=2} and \texorpdfstring{$n>2$}{n>2}}
\label{par:Grrn}

The group $G(r,r,2)$ is the dihedral group $I_2(r)$, which we have already treated in Section~\ref{sec:dihedralgroup}. We assume here that $n > 2$.
Now, the Coxeter number for $G(r,r,n)$ is $h = (n-1)r$. Let $\zeta$ be a primitive $h$-th root of unity, and let $\xi = \zeta^{n-1}$, which is a primitive~$r$-th root of unity.
The group $G(r,r,n)$ is a subgroup of $G(r,1,n)$, and we reuse the notation defined in the previous section.

\medskip

Consider the element $c=c_0\wprod (0,0,\dots,0,1,r-1) \in G(r,r,n)$, where $c_0 \in \Sn_n$ is given by the cycle $(1,2,\dots,n-1)(n)$. We claim that $c$ is a Coxeter element of $G(r,r,n)$. To see that, from the  discussion preceding Section~\ref{sec:classification}, it is enough to prove that $c$ has an eigenvalue equal to $\zeta$. This is true, since the fact that $\xi = \zeta^{n-1}$ implies that the element $c$ has $n-1$ eigenvalues given by $\big\{ \zeta,\zeta^{r+1},\zeta^{2r+1},\ldots,\zeta^{(n-2)r+1}\big\}$, while the~$n$-th eigenvalue is given by $\zeta^{(n-1)(r-1)}$.
\begin{remark}
In the rest of the proof for the group $G(r,r,n)$ we will work with the particular Coxeter element $c=c_0\wprod (0,0,\dots,0,1,r-1)$. This is enough to prove Theorem~\ref{thm:main} for \emph{all} Coxeter elements of the group $G(r,r,n)$, for the same reasons as described in Remark~\ref{rem:coxeterClasses}.
\end{remark}

\step{1}
The description of irreducible characters is now more complicated than for the group $G(r,1,n)$.
Therefore we consider here only the irreducible characters that are of interest to us in the proof of Theorem~\ref{thm:main}.
We refer to Appendix~\ref{sec:irreduciblecharactersrrn} for more information on irreducible characters of $G(r,r,n)$, and for a proof of the following lemma.
\begin{lemma}\label{lemma:step1Grr}
  Let $\chi$ be an irreducible character of $G(r,r,n)$ not vanishing on $c^{-1}$. Then $\chi$ is equal to the restriction of the character $\chi_{\vec\lambda}$ of $G(r,1,n)$ defined by $\eqref{eq:charGr1n}$, where $\vec\lambda$ is one of the following partition vectors.
  \begin{itemize}
    \item $\vec\lambda=(\parA)$ or $\vec\lambda = (\parB)$, where the dots
    denote a list of $(r-1)$ empty partitions;
    \item $\vec\lambda = (\parC)$ for some $1\leq k \leq n-3$, 
    where $\qhook{n}{k}$ denotes the \defn{quasi-hook} $\qhook{n}{k}:=
    [n-k-1,2,1^{k-1}]$, and where the dots
    denote a list of $(r-1)$ empty partitions;
    \item $\vec\lambda=(\parD)$ for some $1\leq j < r$ and $0 \leq k \leq n-2$. Here $\hook{n-1}{k}$ is a hook of size $n-1$ and the partition of size $1$ appears in position $j$. As before the  dots denote lists of empty partitions, respectively of lengths $j-1$ and $r-j-1$.
  \end{itemize}
\end{lemma}

\step{2} This step is encapsulated in the following lemma.
\begin{lemma}\label{lemma:step2Grr}
  The dimension and evaluation on the inverse $c^{-1}$ of the Coxeter element of the characters listed in
  Lemma~\ref{lemma:step1Grr} are given as follows.
  \begin{align*}
    \dim(\parA) = 1, &\quad \chi_{\parA}(c^{-1}) = 1,  \\
    \dim(\parB) = 1, &\quad \chi_{\parB}(c^{-1}) = (-1)^n, \\
    \dim(\parC) = \frac{(n-2-k)k}{n-1}\binom{n}{k+1}, &\quad \chi_{\parC}(c^{-1}) = (-1)^k, \\
    \dim(\parD) = n \cdot \binom{n-2}{k}, &\quad \chi_{\parD}(c^{-1}) = (-1)^k \cdot \xi^{-j}.
  \end{align*}
\end{lemma}
\begin{proof}
If $\lambda$ is a partition of~$n$, observe that in the sum~\eqref{eq:charGr1n}
defining $\chi_{\lambda,\dots}$, all the elements $s\in G(r,1,n)$
contribute. Therefore one has $\chi_{\lambda,\dots}(w) =
\chi_\lambda(|w|)$ for all $w\in G(r,1,n)$.
 This settles the first three cases considered in the lemma, using the
Murnaghan-Nakayama rule and the hook-length formula (the only non immediate fact here is that $\dim(\qhook{n}{k}) = \frac{(n-2-k)k}{n-1}\binom{n}{k+1}$; it follows by a simple computation using the hook-length formula, that we leave to the reader).

In the fourth case we have to be more careful about the elements $s$ contributing to the sum~\eqref{eq:charGr1n}. When evaluating $\chi_{\parD}$ on $c^{-1}$, an element $s\in G(r,1,n)$ contributes to the sum if and only if $s$ belongs to the group $B$ of block-diagonal matrices defined in Proposition~\ref{prop:Gr1n}.
From this we obtain
$$\chi_{\parD}(c^{-1}) = \chi_{\hook{n-1}{k}}((1,2,\dots,n-1))\cdot \xi^{-j},$$
which equals $(-1)^k \cdot \xi^{-j}$ by Lemma~\ref{lemma:hooks}.
When evaluating the same character on the identity, then all the elements $s\in G(r,1,n)$ contribute, so we obtain
$$
\chi_{\parD}(\one)
=
\frac{|G(r,1,n)|}{|B|} \chi_{\hook{n-1}{k}}(\one) 
=
n\dim(\hook{n-1}{k}) = n\binom{n-2}{k},
$$
where we again used Lemma~\ref{lemma:hooks}.
\end{proof}

\step{3} The set of reflections of $G(r,r,n)$ is formed of one conjugacy class,
namely the class $\classR_0$ defined in Section~\ref{par:Gr1n}. The following
lemma gives the evaluation of the normalized irreducible characters at the corresponding
group algebra element $R_0:=\sum_{\tau\in\classR_0}\tau$.
\begin{lemma}
For the characters listed in   Lemma~\ref{lemma:step1Grr}, the normalized character evaluations  
on the sum of all reflections are given by
\begin{align}
  \widetilde\chi_{\parA}(R_0) &=  r\binom{n}{2}, \label{eq:c1} \\
  \widetilde\chi_{\parB}(R_0) &= -r\binom{n}{2}, \label{eq:c2} \\
  \widetilde\chi_{\parC}(R_0) &= \frac{r(n-1)(n-2-2k)}{2}, \label{eq:c3} \\
  \widetilde\chi_{\parD}(R_0) &= \frac{r(n-1)(n-2k-2)}{2}. \label{eq:c4} 
\end{align}
\end{lemma}
\begin{proof}
As already observed in the proof of Lemma~\ref{lemma:step2Grr}, for $\lambda$ a partition of~$n$ we have $\chi_{\lambda, \dots}(w) = \chi_\lambda (|w|)$ for all $w$. Therefore we have 
$\chi_{\lambda,\dots}(R_0) = r\binom{n}{2} \chi_\lambda(|\tau|)$ where $\tau$ is any element of $\classR_0$. Since $|\tau|$ is a transposition, the three first cases \eqref{eq:c1}, \eqref{eq:c2}, \eqref{eq:c3} therefore follow from Lemma~\ref{lemma:contents} (we leave to the reader the details of the computation of the sum of contents of the quasi-hook).

Now let us consider the evaluation $\chi_{\parD}(\tau)$ where $\tau \in \classR_0$. Note that $|\tau|$ is a transposition, say $|\tau| = (i,j)$.
In~\eqref{eq:charGr1n}, an element $s\in G(r,1,n)$ is such that $s^{-1}\tau s \in B$ if and only if $|s^{-1}(n)|\not\in \{i,j\}$. There are $r^n(n-2)(n-1)!$ such elements in $G(r,1,n)$. Therefore~\eqref{eq:charGr1n} evaluates as
\begin{align*}
  \chi_{\parD}(\tau)
  &= \frac{r^n(n-2)(n-1)!}{r^n (n-1)!}
  \chi_{\hook{n-1}{k}}(|\tau|)
  \\
  &= (n-2) \chi_{\hook{n-1}{k}}(|\tau|)
  \\
  &=
  (n-2k-2)\dim(\hook{n-1}{k}),
\end{align*}
where we used~\eqref{eq:hooktranspo} in the last equality. \eqref{eq:c4} then follows from the equality $\dim(\hook{n-1}{k}) = n\dim(\chi_{\parD})$.
\end{proof}

\step{4} From Steps~1--3, Formula~\eqref{eq:frobeniusZt} rewrites as
\begin{align*}
  |W|\cdot \formula_W(t)
  &=
  e^{tr\binom{n}{2}}
  +(-1)^n e^{-tr\binom{n}{2}}
  \\
  &
  + \sum_{k=1}^{n-3} \frac{(n-2-k)k}{n-1} \binom{n}{k+1} (-1)^k
  e^{tr\frac{(n-1)(n-2-2k)}{2}}
  \\
  & + \sum_{k=0}^{n-2} n\binom{n-2}{k} (-1)^k e^{tr\frac{(n-1)(n-2-2k)}{2}}
  \left(\sum_{j=1}^{r-1}\xi^{-j}\right).
\end{align*}
To simplify this expression, first note that the range of summation in the first sum may be extended to $k\in[0,n-2]$ since boundary values vanish. Then, note that $\sum_{j=1}^{r-1}\xi^{-j} = -1$.
Finally, observe that we have
$$
n \binom{n-2}{k} - \frac{(n-2-k)k}{n-1} \binom{n}{k+1} = \binom{n}{k+1}.
$$
Therefore the two sums can be merged together and we obtain
\begin{align*}
  \formula_W(t)
  &= \frac{1}{|W|} \Bigg( e^{t r\binom{n}{2}} +(-1)^n e^{-t r\binom{n}{2}} \\
  & \hspace{50pt}
  + \sum_{k=0}^{n-2} \binom{n}{k+1} (-1)^{k+1} e^{t r\frac{(n-1)(n-2k-2)}{2}} \Bigg).
\end{align*}
We can now finish the proof of Theorem~\ref{thm:main} for the group $G(r,r,n)$.
\begin{proof}[Proof of Theorem~\ref{thm:main} for the group $W = G(r,r,n)$ with $r\geq 2$ and $n>2$]
\hfill Notice that the two isolated terms in the last equality may be incorporated to the sum by extending the summation range to $k\in[-1,n-1]$.
Making the change of index $k\rightarrow k+1$ we obtain
\begin{align*}
\formula_W(t)
  &= 
    \frac{1}{|W|}\sum_{k=0}^{n} \binom{n}{k} (-1)^{k} e^{t r\frac{(n-1)(n-2k)}{2}}\\
    &=
    \frac{1}{|W|}e^{t r\frac{(n-1)}{2}} \left(1-e^{-t r(n-1)}\right)^n \\
    &=
    \frac{1}{|W|}\left(e^{t r\frac{(n-1)}{2}} - e^{-t r\frac{(n-1)}{2}}\right)^n.
\end{align*}
Since we have for $W = G(r,r,n)$ that $|\classR| = |\classR^*| = r\binom{n}{2}$, this coincides with~\eqref{eq:coxetergenus} in this case.
\end{proof}

\section{Proof of the main theorem for the exceptional groups}
\label{sec:exceptionaltypes}

We now discuss the verification of Theorem~\ref{thm:main} for the well-generated exceptional Shephard-Todd classification types.
Those were listed in Table~\ref{tab:exceptionaltypes}.
In Table~\ref{tab:exceptionaldegreescodegrees} in Appendix~\ref{app:exceptionaltypes}, we summarize more detailed information concerning reflections, degrees, codegrees, and number of irreducible representations.

The order, degrees, and codegrees are taken from~\cite{BMR1995}, while the number of irreducible representations were computed using {\tt Chevie}~\cite{CHEVIE}.
This software package does not only provide the number of irreducible representations but also the complete character tables\footnote{see \url{http://www.math.jussieu.fr/~jmichel/gap3/htm/chap073.htm} for examples and a detailed description.}.
The irreducible characters are indexed in {\tt Chevie} by their \defn{degrees} given by the dimension of the corresponding irreducible representation, $\deg(\chi) = \chi(\one)$, together with the smallest integer $k$ for which the irreducible representation $V_\chi$ occurs within the $k$-th symmetric power of the reflection representation $V$,
$$\occ{\chi} := \min\big\{ k : V_\chi \text{ is a summand of } \Sym^k(V) \big\}.$$
Moreover, in {\tt Chevie}, for each character value at a conjugacy class, a class representative within $W$ is given as a permutation of the roots for $W$.
Its matrix in $\GL(V)$ thus has rows given by the images of the simple roots.
Therefore, one can simply find the class representatives of the class $\classCox_\zeta$ given by all Coxeter elements which are $\zeta$-regular, for each primitive $h$-th root of unity  $\zeta$ (see the discussion in the paragraph preceding Section~\ref{sec:classification}). Similarly, one can find the class representatives of the conjugacy classes of reflections.
For the latter, we computed as well the class sizes.
To be more precise, we used the {\tt Sage}~\cite{sage} interface to {\tt GAP3} and {\tt Chevie}, together with the {\tt Sage} patch on reflection groups\footnote{see \url{http://trac.sagemath.org/sage_trac/ticket/11187} for further information.} to determine the conjugacy classes of Coxeter elements, and the classes of reflections, and to compute the class sizes and to evaluate the irreducible characters at these classes.

\medskip

As a backup check, we rechecked this data for the irreducible characters using the {\tt Sage} interface to {\tt GAP4} together with the {\tt GAP4} implementation of irreducible characters of permutation groups. As mentioned above, we used here the presentation of the irreducible well-generated complex reflection groups as permutation groups acting on roots.

\medskip

In Table~\ref{fig:G4example}, the computed data for the exceptional complex reflection group $G_4$ is shown as an example. The first two columns provide $\deg(\chi)$ and $\occ{\chi}$ which uniquely determine the given irreducible representation $\chi$ of $G_4$. We can read from Table~\ref{tab:exceptionaldegreescodegrees} that the Coxeter number for $G_4$ is $6$. We thus have two primitive $6$-th roots of unity, namely $\zeta = e^{2\pi i/6}$ and $\zeta^{-1} = e^{10\pi i/ 6}$. Columns $3$ and $4$ in Figure~\ref{fig:G4example} show the evaluation $\chi(c)$ and $\chi(c^{-1})$ of the irreducible characters at representatives of the two conjugacy classes $\classCox_\zeta$ and $\classCox_{\zeta^{-1}}$. The last column finally gives the evaluation $\chi(R) = \sum_{\tau \in \classR} \chi(\tau)$ of the irreducible characters at the sum of the reflections.
Evaluating the finite sum~\eqref{eq:frobeniusZt} with this data, one easily checks that  Theorem~\ref{thm:main} holds for the group $W=G_4$.

\begin{table}
  \centering
  \begin{tabular}{c | c || c | c | c}
    $\deg(\chi)$ & $\occ{\chi}$ & $\chi(c)$ & $\chi(c^{-1})$ & $\chi(R)$
    \tabularnewline
    \hline
    \hline
    $1$ & $0$ & $1$ & $1$ & $8$
    \tabularnewline
    \hline
    $1$ & $4$ & $\zeta_{3}^{2}$ & $\zeta_{3}$ & $-4$
    \tabularnewline
    \hline
    $1$ & $8$ & $\zeta_{3}$  & $\zeta_{3}^{2}$ & $-4$
    \tabularnewline
    \hline
    $2$ & $1$ & $\zeta_{3}$ & $\zeta_{3}^{2}$ & $4$
    \tabularnewline
    \hline
    $2$ & $3$ & $\zeta_{3}^{2}$ & $\zeta_{3}$ & $4$
    \tabularnewline
    \hline
    $2$ & $5$ & $1$ & $1$ & $-8$
    \tabularnewline
    \hline
    $3$ & $2$ & $0$ & $0$ & $0$
    \tabularnewline
  \end{tabular}
  \caption{Character values for $G_{4}$}
  \label{fig:G4example}
\end{table}

For the interested reader, the complete and long list of the evaluations for all exceptional well-generated complex reflection groups can be found in Tables~\ref{tab:first}--\ref{tab:last} in Appendix~\ref{app:exceptionaltypes}. To save space, we only show the evaluation of the irreducible characters at a representative $c$ of the conjugacy class $\classCox_\zeta$ for $\zeta = e^{2\pi i/h}$, and omit the other primitive $h$-th roots of unity. Nonetheless, we checked the other classes as well.

\medskip

The procedure described in this section allowed us to compute Formula~\eqref{eq:frobeniusZt} for all well-generated exceptional Shephard-Todd classification types, and for all conjugacy classes of Coxeter elements, and thus to verify the following theorem which completes the proof of Theorem~\ref{thm:main}.
\begin{theorem}
  Theorem~\ref{thm:main} holds for all well-generated exceptional Shephard-Todd classification types.
\end{theorem}

\newpage 
\appendix
\addtocontents{toc}{\protect\setcounter{tocdepth}{1}}

\section{Irreducible characters of \texorpdfstring{$G(r,1,n)$}{G(r,1,n)} and \texorpdfstring{$G(r,r,n)$}{G(r,r,n)}.}
\label{app:A}

In this appendix we give the proofs of Proposition~\ref{prop:Gr1n} and
Lemma~\ref{lemma:step1Grr}. We do not claim much originality.
The proof of Proposition~\ref{prop:Gr1n} is easily found in the literature, see
e.g. \cite[Appendix~B]{Macdonald}.
However, for the group $G(r,r,n)$, we were unable to provide references where the
characters are described explicitly enough so that Lemma~\ref{lemma:step1Grr} could be
given as an immediate corollary\footnote{In~\cite{Stembridge}, a description of the irreducible representations of $G(r,r,n)$ is given in a way that is essentially equivalent to ours.}.
Therefore we include its proof here, and since it follows from the same general
theory and helps setting up notation, we also include the proof of Proposition~\ref{prop:Gr1n}.
Both proofs are exercises in elementary representation theory, and only use the description of $G(r,p,n)$ as the semi-direct product of $\Sn_n$ by an abelian group, see below.

\subsection{Irreducible characters of semidirect products with abelian groups}\label{par:Serre}

The semidirect product with an abelian group is a standard situation in
representation theory.
Here for convenience we recall briefly the general theory, taken verbatim from~\cite[Section~8.2]{Serre}, to which we refer for the proofs. We follow the notation in this reference, except that we denote by $K$ the group denoted by $H$ in \cite{Serre}, and by $\kappa$, instead of $h$, elements of this group.

We let $G,A,\Hgroup$ be finite groups such that $A$ is abelian, $G=A \cdot \Hgroup$, and $A
\cap \Hgroup =\{1\}$. We let $X=\mathrm{Hom}(A ,\CC^*)$ be the group of complex
irreducible characters of~$A$. The group $G$ operates on $X$ by $g x(a)
:= x(g^{-1} a g)$ for $g\in G, x\in X, a\in A$. We fix $(x_\iota)_{\iota \in X/\Hgroup}$ be a complete system
of representative of the orbits of $\Hgroup$ in $X$, and for $\iota \in X/\Hgroup$, we
let $\Hgroup_\iota$ be the subgroup of $\Hgroup$ formed by elements $\h$ such that $\h
x_\iota=x_\iota$. We let $G_\iota = A \cdot \Hgroup_\iota$ be the corresponding
subgroup of $G$.
We extend $x_\iota$ to $G_\iota$ by defining $x_\iota(a\cdot \h) = x_\iota(a)$ for $a\in A, \h \in \Hgroup_\iota$.
this defines a one-dimensional character of $G_\iota$. Now, let $\rho$ be an
irreducible representation of $\Hgroup_\iota$. By composing~$\rho$ with the projection
$G_\iota \rightarrow \Hgroup_\iota$, we obtain an irreducible representation $\tilde
\rho$ of $G_\iota$. By considering the tensor product $x_\iota \otimes \tilde\rho$
we obtain an irreducible representation of $G_\iota$.
We let
$\theta_{\iota,\rho}:=\mathrm{Ind}_{G_\iota}^G(x_\iota\otimes\tilde\rho)$ be the corresponding induced representation of $G$.
\begin{proposition}[{\cite[Proposition 25]{Serre}}]\label{prop:Serre}
The representations $\theta_{\iota,\rho}$ form a complete list of irreducible representations
of $G$. Moreover two such representations $\theta_{\iota,\rho}$ and
$\theta_{\iota',\rho'}$ are isomorphic if and only if $\iota=\iota'$ and $\rho$
is isomorphic to $\rho'$.
\end{proposition}

Note that what precedes applies to the group $G=G(r,p,n)$, upon taking $\Hgroup=\Sn_n$, 
and letting
$A$ be the group of diagonal matrices of size~$n$, whose entries are~$r$-th roots
of unity $(\xi_1,\xi_2,\dots,\xi_n)$ such that the product
$\xi_1\xi_2\dots \xi_n$ is an
$(r/p)$-th root of unity.
In the remainder of this section we apply the general theory to the two cases that are of interest to us, namely $p=1$ and $p=r$.
In what follows $\xi$ is a fixed primitive~$r$-th root of unity.

\subsection{Irreducible characters for the group \texorpdfstring{$G(r,1,n)$}{G(r,1,n)}}
\label{sec:irreduciblecharactersr1n}

In the case of $G(r,1,n)$ we have $A\cong
\left(\mathbb{Z}/r\mathbb{Z}\right)^n$, and $X=\mathrm{Hom}(A,\CC^*)$ is formed of the
homomorphisms $\{x_\ell,\ell\in A\}$ given by 
$$x_\ell(a) = \xi^{ \ell_1 a_1 + \ell_2 a_2 + \dots + \ell_n a_n}$$
for $a \in A$.
The orbits of $X$ under the action of $\Hgroup=\Sn_n$ are indexed  by vectors $\iota=(i_0,i_1,\dots ,i_{r-1})$ of sum~$n$, think of $i_m$ as the number of coordinates of $\ell$ equal to $m$.
For a given vector $\iota$ there is a natural representative $x_\iota =
x_{(0^{i_0},1^{i_1},\dots)}$, and we obtain this way a complete system of representatives of $X/\Hgroup$.
The subgroup $\Hgroup_\iota$ of $\Sn_n$ formed by elements $\h$ such that $\h x_\iota =
x_\iota$ is, clearly, the group of block-diagonal matrices
$$\Hgroup_\iota=\Sn_{i_0}\times\Sn_{i_1}\times \dots \times \Sn_{i_{r-1}}.$$
The irreducible representations of this group are given by $\rho = \rho_{\lambda_0} \otimes \rho_{\lambda_1}\otimes \dots \otimes \rho_{\lambda_{r-1}}$ where $\vec{\lambda}$ is a vector of partitions as in Proposition~\ref{prop:Gr1n} and $\rho_{\lambda_m}$ is the representation of $\Sn_{i_m}$ indexed by $\lambda_m$, for $0\leq m < r$.
Note that the group $G_\iota = A. \Hgroup_\iota$ is nothing but the group $B$ of block-diagonal 
matrices defined in Proposition~\ref{prop:Gr1n}. 
With the notation of Paragraph~\ref{par:Serre}, the tensor product $x_\iota \otimes\tilde\rho
$ is a representation of $G_\iota$ of character
$$
\chi_{x_\iota\otimes \tilde\rho }(w_0,w_1\dots,w_{r-1}) = 
\prod_{\ell=0}^{r-1} \chi_{\lambda_\ell}(|w_\ell|)\xi^{\ell \cdot \|w_\ell\| }.
$$
The representation $\theta_{\iota,\rho}$ is, by definition, the induced representation $\mathrm{Ind}_{G_\iota}^G(\tilde\rho \otimes x_\iota)$, whose character is given by~\eqref{eq:charGr1n} (for general background on characters of induced representations, see  e.g. \cite[Theorem 12]{Serre}).
Therefore, Proposition~\ref{prop:Serre} becomes Proposition~\ref{prop:Gr1n} is this setting.

\subsection{Irreducible characters for the group \texorpdfstring{$G(r,r,n)$}{G(r,r,n)}}
\label{sec:irreduciblecharactersrrn}

We are going to use the general theory of Paragraph~\ref{par:Serre} to describe the irreducible representations of $G(r,r,n)$ as induced representations, but we will not go as far as in the case of $G(r,1,n)$.
More precisely, in the notation of~\ref{par:Serre}, we will describe explicitly the groups $G_\iota$, but we will not describe explicitly their irreducible representations in general.
Instead we will keep in mind that, to prove Lemma~\ref{lemma:step1Grr}, we are only interested in characters that do not vanish on the inverse $c^{-1}$ of the Coxeter element, so we will disregard any group $G_\iota$ that contains no conjugate of~$c^{-1}$. 
As a side remark, let us mention that we could have proved Lemma~\ref{lemma:step1Grr} directly using Proposition~\ref{prop:Gr1n} and classical criteria regarding restrictions of induced representations (\cite[Chapter 7]{Serre}). However the approach presented here is not longer, and it requires no more background than the one presented in Appendix~\ref{par:Serre}.

Recall that in the case of $G(r,r,n)$ the Coxeter element is $c=c_0\wr (0,\dots,0,1,r-1)$ where $c_0=(1,2,\dots,n-1)(n)$ has cycle type $[n-1,1]$.
As in Section~\ref{par:Grrn} we assume that $n > 2$.

In the case of $G(r,r,n)$ we have $A\cong \mathrm{Ker}\ \pi$ where $\pi: \left(\mathbb{Z}/r\mathbb{Z}\right)^n  \rightarrow \mathbb{Z}/r\mathbb{Z}$ is the sum-of-coordinate mapping.
Therefore $X=\mathrm{Hom}(A,\CC^*)$ is formed of the homomorphisms
$$x_\ell : a \in A \longmapsto \xi^{ \ell_1 a_1 + \ell_2 a_2 + \dots \ell_n a_n}$$
for $\ell\in (\mathbb{Z}/r\mathbb{Z})^n /\!\! \approx$, where $\ell \approx \ell'$ if and only if $\ell-\ell'$ is of the form $(z,z,\dots,z)$ for some $z\in \mathbb{Z}/r\mathbb{Z}$.
The orbits of $X$ under the action of $\Hgroup=\Sn_n$ are indexed by vectors $\iota = (i_0,i_1,\dots,i_{r-1})$ of total sum~$n$ considered up to circular permutation (think of $i_m$ as the number of coordinates of $\ell$ equal to $m$, and observe that a translation on $\ell$ by a vector of the form $(z,z,\dots,z)$ acts as a circular permutation on $\iota$).
For each such orbit we fix an arbitrary representative $x_\iota = x_{0^{i_0}1^{i_1}\dots}$ in order to obtain a complete system of representatives of $X/\Hgroup$.

Fix $\iota=(i_0,i_1,\dots, i_{r-1})$ and let $d$ be the smallest positive integer such that $i_{\ell+d}=i_\ell$ for all $\ell$, indices being taken modulo~$r$. Note that $d$ is a divisor of~$r$, say $r=md$.
We now determine the group $\Hgroup_\iota$, in the notation of~\ref{par:Serre}.
By definition, an element $\h \in \mathfrak{S}_n$ is such that $\h x_\iota = x_\iota$ if and only if
$$\h (0^{i_0},1^{i_1},\dots,(r-1)^{i_{r-1}}) =  (0^{i_0},1^{i_1},\dots,(r-1)^{i_{r-1}})
\text{ modulo $\approx$,}
$$
where $\approx$ is the equivalence relation defined above, i.e. iff there exists 
$k_\h\in \mathbb{Z}/r\mathbb{Z}$ such that
\begin{align}\label{eq:stabH}
\h (0^{i_0},1^{i_1},\dots,(r-1)^{i_{r-1}}) =
(k_\h^{i_0},(k_\h+1)^{i_1},\dots,(k_\h+r-1)^{i_{r-1}}).
\end{align}
Now, let $I_\ell$ be the integer interval $I_\ell=\llbracket i_0 + i_1 + \dots +i_{\ell-1}+1, i_0 + i_1 + \dots + i_\ell\rrbracket$ for $0\leq \ell <r$. Note that $\llbracket 1,r \rrbracket = I_0 \uplus I_1 \uplus \dots \uplus I_{r-1}$.
Then \eqref{eq:stabH} implies that for each $0\leq \ell <r$, we have:
\begin{align}\label{eq:stabH2}
\h ( I_\ell ) = I_{\ell+k_\h},
\end{align}
indices taken modulo~$r$. In particular, we have that $i_{\ell}=i_{\ell+k_\h}$ for all $\ell$, so that $k_\h$ is a multiple of $d$.

\begin{lemma}
Unless one of the entries of $\iota$ is equal to~$n$ or $n-1$, the inverse $c^{-1}$ of the Coxeter element has no conjugates in the group $G_\iota$.
\end{lemma}
\begin{proof}
  First note that $c^{-1}$ has a conjugate in $G_\iota$ if and only if $c_0^{-1}=|c^{-1}|$ has a conjugate in $\Hgroup_\iota$. Since $c_0^{-1}=(n-1,n-2,\dots,1)(n)$ has a fixed point, then by~\eqref{eq:stabH2}, if such a conjugate $\h$ exists it is such that $k_\h=0$, keeping previous notation. Hence using~\eqref{eq:stabH2} again, the conjugate belongs to the group of block-diagonal matrices $\Sn_{i_0}\times\Sn_{i_1}\times\dots \times \Sn_{i_{r-1}}$.
  But since $c_0^{-1}$ has cycle type $[n-1,1]$, one of the blocks must be of size $n-1$ or~$n$, in order to fit the cycle of length $(n-1)$.
\end{proof}
Recalling the classical expression of induced characters as a sum over conjugates, see e.g. \cite[Theorem 12]{Serre}, this implies the following lemma.
\begin{lemma}
  Unless one of the entries of $\iota$ is equal to~$n$ or $n-1$, the character of the representation $\theta_{\iota,\rho}$ defined in \ref{par:Serre} vanishes on the inverse $c^{-1}$ of the Coxeter element, for any irreducible representation $\rho$ of $\Hgroup_\iota$.
\end{lemma}
Since we only want to prove Lemma~\ref{lemma:step1Grr}, we may thus focus only on the case when $\iota$ has a coordinate equal to~$n$ or $n-1$.
We treat both cases separately.

\medskip

In the first case, we have that $\iota = (n, 0,\dots,0)$ up to circular permutation.
The group $\Hgroup_\iota$ coincides with $\Sn_n$, and $G_\iota$ coincides with $G(r,r,n)$.
For any irreducible representation $\rho=\rho_\lambda$ of $\Sn_n$, the character of the irreducible representation $\theta_{\iota,\rho}$ on an element $w\in G(r,r,n)$ is given, simply, by the $\Sn_n$-character $\chi_\lambda(|w|)$.

It remains to determine which partitions $\lambda$ of~$n$ are such that the character $\chi_\lambda$ does not vanish on the permutation $c_0^{-1}$ of type $[n-1,1]$.
Clearly, the Murnaghan-Nakayama rule implies than $\lambda$ is a hook or a quasi-hook, since these are the only partitions containing a strip of length $n-1$.
A closer look at the case of hooks shows, moreover, that the only hooks for which the character does not vanish are~$[n]$ and $[1^n]$.
Finally, it is clear that the characters thus defined coincide with the restrictions to $G(r,r,n)$ of the representations $\chi_{[n],\dots}$, $\chi_{[1^n],\dots}$ and $\chi_{\qhook{n}{k},\dots}$ as stated in Lemma~\ref{lemma:step1Grr}.

\medskip
In the second case, we have that, up to a circular permutation, $\iota=(n-1,0,\dots,0,1,0,\dots,0)$ where the \lq\lq$1$\rq\rq\ appears in position $j$, for $1\leq j <r$.
Notice that since $n\neq 2$ we have  $n-1\neq1$ so $\iota$ is fixed by no non-trivial circular permutation. In previous notation, this means that $d=r$, hence all $\h \in \Hgroup_\iota$ are such that $k_\h=0 (mod \ r)$.  Thus by~\eqref{eq:stabH2} the group $\Hgroup_\iota$ is the group of block-diagonal matrices
$$
\Hgroup_\iota = \Sn_{n-1}\times \Sn_1.
$$ 
The group $G_\iota$ is the subgroup of $G(r,r,n)$ formed by matrices $w$
such that $|w|\in \Hgroup_\iota$.
Irreducible representations of $\Hgroup_\iota$ are indexed by partitions $\lambda$ of $n-1$, and for each such partition we may thus construct the representation $\theta_{\iota, \rho}$ of $G(r,r,n)$ as in~\ref{par:Serre}. By definition, its character is
$$
\chi(w) = \sum_{\substack{s \in G(r,r,n) \\ s^{-1}ws \in G_\iota}} \chi_{\lambda}(|w_1|)
\xi^{j \|w_2\|}
$$
where $s^{-1}ws \in B$ in the sum is denoted by $(w_1,w_2)$ where $w_1$ is the principal $(n-1)\times(n-1)$ submatrix and $w_2$ is the bottom-right entry.
Since $|c^{-1}|=c_0^{-1}$ has cycle type $[n-1,1]$, the only conjugates $s^{-1}c^{-1}s$ of $c^{-1}$ that are in $G_\iota$ are such that $|w_1|$ is an $(n-1)$-cycle.
Therefore by the Murnaghan-Nakayama rule one has $\chi(c^{-1})=0$ unless $\lambda$ is a hook of size $n-1$.
This concludes the proof of Lemma~\ref{lemma:step1Grr}, by observing that the character we just described coincides with the restriction of the $G(r,1,n)$-character $\chi_{\lambda,\dots,1,\dots}$ as in the statement of Lemma~\ref{lemma:step1Grr}.

\newpage
\captionsetup{width=.9\textwidth}
\section{Character evaluations for the exceptional groups}
\label{app:exceptionaltypes}

\small

\begin{table}[h]
  \centering
\begin{tabular}{c || c | c | c | c }
    type & order & degrees & codegrees & \small \# irr.rep.\\
    \hline
    \hline
    &&&\\[-10pt]
    $G_4$ & $24$ & $4,6$ & $0,2$ & $7$ \\
    \hline
    &&&\\[-10pt]
    $G_5$ & $72$ & $6,12$ & $0,6$ & $21$ \\
    \hline
    &&&\\[-10pt]
    $G_6$ & $48$ & $4,12$ & $0,8$ & $14$ \\
    \hline
    &&&\\
\end{tabular}
  \caption{\small The order, degrees and codegrees of all well-generated exceptional Shephard-Todd classification types.}
  \label{tab:exceptionaldegreescodegrees}
\end{table}

\twocolumn
\footnotesize

\topcaption{\small Character values for $G_{4}$\vspace{-10pt}}\label{tab:first}
\begin{supertabular}{ v{22px} | v{22px} || v{48px} | v{30px} }
    $\deg(\chi)$ & $\occ{\chi}$ & $\chi(c)$ & $\chi(R)$
    \tabularnewline
    \hline
    \hline
    $1$ & $0$ & $1$ & $8$
    \tabularnewline
    \hline
    $1$ & $4$ & $\zeta_{3}^{2}$ & $-4$
    \tabularnewline
    \hline
    $1$ & $8$ & $\zeta_{3}$ & $-4$
    \tabularnewline
    \hline
    $2$ & $1$ & $\zeta_{3}$ & $4$
    \tabularnewline
    \hline
    $2$ & $3$ & $\zeta_{3}^{2}$ & $4$
    \tabularnewline
    \hline
    $2$ & $5$ & $1$ & $-8$
    \tabularnewline
    \hline
    $3$ & $2$ & $0$ & $0$
    \tabularnewline
\end{supertabular}

\bigskip\bigskip\bigskip

\topcaption{\small Character values for $G_{5}$\vspace{-10pt}}

\begin{supertabular}{ v{22px} | v{22px} || v{48px} | v{30px} }
    $\deg(\chi)$ & $\occ{\chi}$ & $\chi(c)$ & $\chi(R)$
    \tabularnewline
    \hline
    \hline
    $1$ & $0$ & $1$ & $16$
    \tabularnewline
    \hline
    $1$ & $4$ & $\zeta_{3}$ & $4$
    \tabularnewline
    \hline
    $1$ & $4$ & $\zeta_{3}$ & $4$
    \tabularnewline
    \hline
    $1$ & $8$ & $\zeta_{3}^{2}$ & $4$
    \tabularnewline
    \hline
    $1$ & $8$ & $\zeta_{3}^{2}$ & $-8$
    \tabularnewline
    \hline
    $1$ & $8$ & $\zeta_{3}^{2}$ & $4$
    \tabularnewline
    \hline
    $1$ & $12$ & $1$ & $-8$
    \tabularnewline
    \hline
    $1$ & $12$ & $1$ & $-8$
    \tabularnewline
    \hline
    $1$ & $16$ & $\zeta_{3}$ & $-8$
    \tabularnewline
    \hline
    $2$ & $1$ & $0$ & $8$
    \tabularnewline
    \hline
    $2$ & $3$ & $0$ & $8$
    \tabularnewline
    \hline
    $2$ & $3$ & $0$ & $8$
    \tabularnewline
    \hline
    $2$ & $5$ & $0$ & $-4$
    \tabularnewline
    \hline
    $2$ & $5$ & $0$ & $8$
    \tabularnewline
    \hline
    $2$ & $5$ & $0$ & $-4$
    \tabularnewline
    \hline
    $2$ & $7$ & $0$ & $-4$
    \tabularnewline
    \hline
    $2$ & $7$ & $0$ & $-4$
    \tabularnewline
    \hline
    $2$ & $9$ & $0$ & $-16$
    \tabularnewline
    \hline
    $3$ & $2$ & $-\zeta_{3}^{2}$ & $0$
    \tabularnewline
    \hline
    $3$ & $4$ & $-\zeta_{3}$ & $0$
    \tabularnewline
    \hline
    $3$ & $6$ & $-1$ & $0$
    \tabularnewline
\end{supertabular}

\bigskip\bigskip\bigskip

\topcaption{\small Character values for $G_{6}$\vspace{-10pt}}
\begin{supertabular}{ v{22px} | v{22px} || v{48px} | v{30px} }
    $\deg(\chi)$ & $\occ{\chi}$ & $\chi(c)$ & $\chi(R)$
    \tabularnewline
    \hline
    \hline
    $1$ & $0$ & $1$ & $14$
    \tabularnewline
    \hline
    $1$ & $4$ & $\zeta_{3}$ & $2$
    \tabularnewline
    \hline
    $1$ & $6$ & $-1$ & $2$
    \tabularnewline
    \hline
    $1$ & $8$ & $\zeta_{3}^{2}$ & $2$
    \tabularnewline
    \hline
    $1$ & $10$ & $-\zeta_{3}$ & $-10$
    \tabularnewline
    \hline
    $1$ & $14$ & $-\zeta_{3}^{2}$ & $-10$
    \tabularnewline
    \hline
    $2$ & $1$ & $\zeta_{12}^{11}$ & $4$
    \tabularnewline
    \hline
    $2$ & $3$ & $-\zeta_{12}^{11}$ & $4$
    \tabularnewline
    \hline
    $2$ & $3$ & $-\zeta_{12}^{7}$ & $4$
    \tabularnewline
    \hline
    $2$ & $5$ & $\zeta_{12}^{7}$ & $4$
    \tabularnewline
    \hline
    $2$ & $5$ & $\zeta_{4}$ & $-8$
    \tabularnewline
    \hline
    $2$ & $7$ & $-\zeta_{4}$ & $-8$
    \tabularnewline
    \hline
    $3$ & $2$ & $0$ & $6$
    \tabularnewline
    \hline
    $3$ & $4$ & $0$ & $-6$
    \tabularnewline
\end{supertabular}

\bigskip\bigskip\bigskip

\topcaption{\small Character values for $G_{8}$\vspace{-10pt}}
\begin{supertabular}{ v{22px} | v{22px} || v{48px} | v{30px} }
    $\deg(\chi)$ & $\occ{\chi}$ & $\chi(c)$ & $\chi(R)$
    \tabularnewline
    \hline
    \hline
    $1$ & $0$ & $1$ & $18$
    \tabularnewline
    \hline
    $1$ & $6$ & $-1$ & $-6$
    \tabularnewline
    \hline
    $1$ & $12$ & $1$ & $-6$
    \tabularnewline
    \hline
    $1$ & $18$ & $-1$ & $-6$
    \tabularnewline
    \hline
    $2$ & $1$ & $\zeta_{4}$ & $12$
    \tabularnewline
    \hline
    $2$ & $4$ & $-1$ & $12$
    \tabularnewline
    \hline
    $2$ & $7$ & $-\zeta_{4}$ & $12$
    \tabularnewline
    \hline
    $2$ & $7$ & $-\zeta_{4}$ & $-12$
    \tabularnewline
    \hline
    $2$ & $10$ & $1$ & $-12$
    \tabularnewline
    \hline
    $2$ & $13$ & $\zeta_{4}$ & $-12$
    \tabularnewline
    \hline
    $3$ & $2$ & $0$ & $6$
    \tabularnewline
    \hline
    $3$ & $4$ & $0$ & $6$
    \tabularnewline
    \hline
    $3$ & $6$ & $0$ & $6$
    \tabularnewline
    \hline
    $3$ & $8$ & $0$ & $-18$
    \tabularnewline
    \hline
    $4$ & $3$ & $\zeta_{4}$ & $0$
    \tabularnewline
    \hline
    $4$ & $5$ & $-\zeta_{4}$ & $0$
    \tabularnewline
\end{supertabular}

\bigskip\bigskip\bigskip

\topcaption{\small Character values for $G_{9}$\vspace{-10pt}}
\begin{supertabular}{ v{22px} | v{22px} || v{48px} | v{30px} }
    $\deg(\chi)$ & $\occ{\chi}$ & $\chi(c)$ & $\chi(R)$
    \tabularnewline
    \hline
    \hline
    $1$ & $0$ & $1$ & $30$
    \tabularnewline
    \hline
    $1$ & $6$ & $\zeta_{4}$ & $6$
    \tabularnewline
    \hline
    $1$ & $12$ & $-1$ & $6$
    \tabularnewline
    \hline
    $1$ & $12$ & $-1$ & $6$
    \tabularnewline
    \hline
    $1$ & $18$ & $-\zeta_{4}$ & $-18$
    \tabularnewline
    \hline
    $1$ & $18$ & $-\zeta_{4}$ & $6$
    \tabularnewline
    \hline
    $1$ & $24$ & $1$ & $-18$
    \tabularnewline
\end{supertabular}

\bigskip\bigskip\bigskip

\topcaption{\small Character values for $G_{10}$\vspace{-10pt}}
\begin{supertabular}{ v{22px} | v{22px} || v{48px} | v{30px} }
    $\deg(\chi)$ & $\occ{\chi}$ & $\chi(c)$ & $\chi(R)$
    \tabularnewline
    \hline
    \hline
    $1$ & $0$ & $1$ & $34$
    \tabularnewline
    \hline
    $1$ & $6$ & $\zeta_{4}$ & $10$
    \tabularnewline
    \hline
    $1$ & $8$ & $\zeta_{3}$ & $10$
    \tabularnewline
    \hline
    $1$ & $12$ & $-1$ & $10$
    \tabularnewline
    \hline
    $1$ & $14$ & $\zeta_{12}^{7}$ & $-14$
    \tabularnewline
    \hline
    $1$ & $16$ & $\zeta_{3}^{2}$ & $10$
    \tabularnewline
    \hline
    $1$ & $18$ & $-\zeta_{4}$ & $10$
    \tabularnewline
\end{supertabular}

\bigskip\bigskip\bigskip

\topcaption{\small Character values for $G_{14}$\vspace{-10pt}}
\begin{supertabular}{ v{22px} | v{22px} || v{48px} | v{30px} }
    $\deg(\chi)$ & $\occ{\chi}$ & $\chi(c)$ & $\chi(R)$
    \tabularnewline
    \hline
    \hline
    $1$ & $0$ & $1$ & $28$
    \tabularnewline
    \hline
    $1$ & $8$ & $\zeta_{3}$ & $4$
    \tabularnewline
    \hline
    $1$ & $12$ & $-1$ & $4$
    \tabularnewline
    \hline
    $1$ & $16$ & $\zeta_{3}^{2}$ & $4$
    \tabularnewline
    \hline
    $1$ & $20$ & $-\zeta_{3}$ & $-20$
    \tabularnewline
    \hline
    $1$ & $28$ & $-\zeta_{3}^{2}$ & $-20$
    \tabularnewline
    \hline
    $2$ & $1$ & $\zeta_{24} + \zeta_{24}^{19}$ & $8$
    \tabularnewline
    \hline
    $2$ & $4$ & $0$ & $8$
    \tabularnewline
    \hline
    $2$ & $5$ & $-\zeta_{24}^{11} - \zeta_{24}^{17}$ & $8$
    \tabularnewline
    \hline
    $2$ & $7$ & $-\zeta_{24} - \zeta_{24}^{19}$ & $8$
    \tabularnewline
    \hline
    $2$ & $8$ & $0$ & $8$
    \tabularnewline
    \hline
    $2$ & $9$ & $\zeta_{8} + \zeta_{8}^{3}$ & $-16$
    \tabularnewline
    \hline
    $2$ & $11$ & $\zeta_{24}^{11} + \zeta_{24}^{17}$ & $8$
    \tabularnewline
    \hline
    $2$ & $12$ & $0$ & $-16$
    \tabularnewline
    \hline
    $2$ & $15$ & $-\zeta_{8} - \zeta_{8}^{3}$ & $-16$
    \tabularnewline
    \hline
    $3$ & $2$ & $-\zeta_{3}$ & $12$
    \tabularnewline
    \hline
    $3$ & $4$ & $-\zeta_{3}^{2}$ & $12$
    \tabularnewline
    \hline
    $3$ & $6$ & $-1$ & $12$
    \tabularnewline
    \hline
    $3$ & $6$ & $1$ & $-12$
    \tabularnewline
    \hline
    $3$ & $8$ & $\zeta_{3}$ & $-12$
    \tabularnewline
    \hline
    $3$ & $10$ & $\zeta_{3}^{2}$ & $-12$
    \tabularnewline
    \hline
    $4$ & $3$ & $0$ & $16$
    \tabularnewline
    \hline
    $4$ & $5$ & $0$ & $-8$
    \tabularnewline
    \hline
    $4$ & $7$ & $0$ & $-8$
    \tabularnewline
\end{supertabular}

\bigskip\bigskip\bigskip

\topcaption{\small Character values for $G_{16}$\vspace{-10pt}}
\begin{supertabular}{ v{22px} | v{22px} || v{48px} | v{30px} }
    $\deg(\chi)$ & $\occ{\chi}$ & $\chi(c)$ & $\chi(R)$
    \tabularnewline
    \hline
    \hline
    $1$ & $0$ & $1$ & $48$
    \tabularnewline
    \hline
    $1$ & $12$ & $\zeta_{5}^{2}$ & $-12$
    \tabularnewline
    \hline
    $1$ & $24$ & $\zeta_{5}^{4}$ & $-12$
    \tabularnewline
    \hline
    $1$ & $36$ & $\zeta_{5}$ & $-12$
    \tabularnewline
    \hline
    $1$ & $48$ & $\zeta_{5}^{3}$ & $-12$
    \tabularnewline
    \hline
    $2$ & $1$ & $\zeta_{5}$ & $36$
    \tabularnewline
    \hline
    $2$ & $7$ & $\zeta_{5}^{2}$ & $36$
    \tabularnewline
\end{supertabular}

\bigskip\bigskip\bigskip

\topcaption{\small Character values for $G_{17}$\vspace{-10pt}}
\begin{supertabular}{ v{22px} | v{22px} || v{48px} | v{30px} }
    $\deg(\chi)$ & $\occ{\chi}$ & $\chi(c)$ & $\chi(R)$
    \tabularnewline
    \hline
    \hline
    $1$ & $0$ & $1$ & $78$
    \tabularnewline
    \hline
    $1$ & $12$ & $\zeta_{5}$ & $18$
    \tabularnewline
    \hline
    $1$ & $24$ & $\zeta_{5}^{2}$ & $18$
    \tabularnewline
    \hline
    $1$ & $30$ & $-1$ & $18$
    \tabularnewline
    \hline
    $1$ & $36$ & $\zeta_{5}^{3}$ & $18$
    \tabularnewline
    \hline
    $1$ & $42$ & $-\zeta_{5}$ & $-42$
    \tabularnewline
    \hline
    $1$ & $48$ & $\zeta_{5}^{4}$ & $18$
    \tabularnewline
\end{supertabular}

\bigskip\bigskip\bigskip

\topcaption{\small Character values for $G_{18}$\vspace{-10pt}}
\begin{supertabular}{ v{22px} | v{22px} || v{48px} | v{30px} }
    $\deg(\chi)$ & $\occ{\chi}$ & $\chi(c)$ & $\chi(R)$
    \tabularnewline
    \hline
    \hline
    $1$ & $0$ & $1$ & $88$
    \tabularnewline
    \hline
    $1$ & $12$ & $\zeta_{5}$ & $28$
    \tabularnewline
    \hline
    $1$ & $20$ & $\zeta_{3}$ & $28$
    \tabularnewline
    \hline
    $1$ & $24$ & $\zeta_{5}^{2}$ & $28$
    \tabularnewline
    \hline
    $1$ & $32$ & $\zeta_{15}^{8}$ & $-32$
    \tabularnewline
    \hline
    $1$ & $36$ & $\zeta_{5}^{3}$ & $28$
    \tabularnewline
    \hline
    $1$ & $40$ & $\zeta_{3}^{2}$ & $28$
    \tabularnewline
\end{supertabular}

\bigskip\bigskip\bigskip

\topcaption{\small Character values for $G_{20}$\vspace{-10pt}}
\begin{supertabular}{ v{22px} | v{22px} || v{48px} | v{30px} }
    $\deg(\chi)$ & $\occ{\chi}$ & $\chi(c)$ & $\chi(R)$
    \tabularnewline
    \hline
    \hline
    $1$ & $0$ & $1$ & $40$
    \tabularnewline
    \hline
    $1$ & $20$ & $\zeta_{3}^{2}$ & $-20$
    \tabularnewline
    \hline
    $1$ & $40$ & $\zeta_{3}$ & $-20$
    \tabularnewline
    \hline
    $2$ & $1$ & $-\zeta_{15}^{2} - \zeta_{15}^{8}$ & $20$
    \tabularnewline
    \hline
    $2$ & $7$ & $-\zeta_{15}^{11} - \zeta_{15}^{14}$ & $20$
    \tabularnewline
    \hline
    $2$ & $11$ & $-\zeta_{15}^{7} - \zeta_{15}^{13}$ & $20$
    \tabularnewline
    \hline
    $2$ & $17$ & $-\zeta_{15} - \zeta_{15}^{4}$ & $20$
    \tabularnewline
\end{supertabular}

\bigskip\bigskip\bigskip

\topcaption{\small Character values for $G_{21}$\vspace{-10pt}}
\begin{supertabular}{ v{22px} | v{22px} || v{48px} | v{30px} }
    $\deg(\chi)$ & $\occ{\chi}$ & $\chi(c)$ & $\chi(R)$
    \tabularnewline
    \hline
    \hline
    $1$ & $0$ & $1$ & $70$
    \tabularnewline
    \hline
    $1$ & $20$ & $\zeta_{3}$ & $10$
    \tabularnewline
    \hline
    $1$ & $30$ & $-1$ & $10$
    \tabularnewline
    \hline
    $1$ & $40$ & $\zeta_{3}^{2}$ & $10$
    \tabularnewline
    \hline
    $1$ & $50$ & $-\zeta_{3}$ & $-50$
    \tabularnewline
    \hline
    $1$ & $70$ & $-\zeta_{3}^{2}$ & $-50$
    \tabularnewline
    \hline
    $2$ & $1$ & $-\zeta_{60}^{19} - \zeta_{60}^{31}$ & $20$
    \tabularnewline
\end{supertabular}

\bigskip\bigskip\bigskip

\topcaption{\small Character values for $H3$}
\begin{supertabular}{ v{22px} | v{22px} || v{48px} | v{30px} }
    $\deg(\chi)$ & $\occ{\chi}$ & $\chi(c)$ & $\chi(R)$
    \tabularnewline
    \hline
    \hline
    $1$ & $0$ & $1$ & $15$
    \tabularnewline
    \hline
    $1$ & $15$ & $-1$ & $-15$
    \tabularnewline
    \hline
    $3$ & $1$ & $\zeta_{5} + \zeta_{5}^{4}$ & $15$
    \tabularnewline
    \hline
    $3$ & $3$ & $\zeta_{5}^{2} + \zeta_{5}^{3}$ & $15$
    \tabularnewline
    \hline
    $3$ & $6$ & $-\zeta_{5} - \zeta_{5}^{4}$ & $-15$
    \tabularnewline
    \hline
    $3$ & $8$ & $-\zeta_{5}^{2} - \zeta_{5}^{3}$ & $-15$
    \tabularnewline
    \hline
    $4$ & $3$ & $1$ & $0$
    \tabularnewline
    \hline
    $4$ & $4$ & $-1$ & $0$
    \tabularnewline
    \hline
    $5$ & $2$ & $0$ & $15$
    \tabularnewline
    \hline
    $5$ & $5$ & $0$ & $-15$
    \tabularnewline
\end{supertabular}

\bigskip\bigskip\bigskip

\topcaption{\small Character values for $G_{24}$\vspace{-10pt}}
\begin{supertabular}{ v{22px} | v{22px} || v{48px} | v{30px} }
    $\deg(\chi)$ & $\occ{\chi}$ & $\chi(c)$ & $\chi(R)$
    \tabularnewline
    \hline
    \hline
    $1$ & $0$ & $1$ & $21$
    \tabularnewline
    \hline
    $1$ & $21$ & $-1$ & $-21$
    \tabularnewline
    \hline
    $3$ & $1$ & $-\zeta_{7} - \zeta_{7}^{2} - \zeta_{7}^{4}$ & $21$
    \tabularnewline
    \hline
    $3$ & $3$ & $-\zeta_{7}^{3} - \zeta_{7}^{5} - \zeta_{7}^{6}$ & $21$
    \tabularnewline
    \hline
    $3$ & $8$ & $\zeta_{7} + \zeta_{7}^{2} + \zeta_{7}^{4}$ & $-21$
    \tabularnewline
    \hline
    $3$ & $10$ & $\zeta_{7}^{3} + \zeta_{7}^{5} + \zeta_{7}^{6}$ & $-21$
    \tabularnewline
    \hline
    $6$ & $2$ & $-1$ & $42$
    \tabularnewline
    \hline
    $6$ & $9$ & $1$ & $-42$
    \tabularnewline
    \hline
    $7$ & $3$ & $0$ & $21$
    \tabularnewline
    \hline
    $7$ & $6$ & $0$ & $-21$
    \tabularnewline
    \hline
    $8$ & $4$ & $1$ & $0$
    \tabularnewline
    \hline
    $8$ & $5$ & $-1$ & $0$
    \tabularnewline
\end{supertabular}

\bigskip\bigskip\bigskip

\topcaption{\small Character values for $G_{25}$\vspace{-10pt}}
\begin{supertabular}{ v{22px} | v{22px} || v{48px} | v{30px} }
    $\deg(\chi)$ & $\occ{\chi}$ & $\chi(c)$ & $\chi(R)$
    \tabularnewline
    \hline
    \hline
    $1$ & $0$ & $1$ & $24$
    \tabularnewline
    \hline
    $1$ & $12$ & $1$ & $-12$
    \tabularnewline
    \hline
    $1$ & $24$ & $1$ & $-12$
    \tabularnewline
    \hline
    $2$ & $3$ & $0$ & $12$
    \tabularnewline
    \hline
    $2$ & $9$ & $0$ & $12$
    \tabularnewline
    \hline
    $2$ & $15$ & $0$ & $-24$
    \tabularnewline
    \hline
    $3$ & $1$ & $\zeta_{3}$ & $36$
    \tabularnewline
    \hline
    $3$ & $5$ & $\zeta_{3}^{2}$ & $36$
    \tabularnewline
    \hline
    $3$ & $5$ & $\zeta_{3}^{2}$ & $0$
    \tabularnewline
    \hline
    $3$ & $6$ & $-1$ & $0$
    \tabularnewline
    \hline
    $3$ & $13$ & $\zeta_{3}$ & $-36$
    \tabularnewline
    \hline
    $3$ & $13$ & $\zeta_{3}$ & $0$
    \tabularnewline
    \hline
    $3$ & $17$ & $\zeta_{3}^{2}$ & $-36$
    \tabularnewline
    \hline
    $6$ & $2$ & $0$ & $36$
    \tabularnewline
    \hline
    $6$ & $4$ & $0$ & $36$
    \tabularnewline
    \hline
    $6$ & $4$ & $0$ & $0$
    \tabularnewline
    \hline
    $6$ & $8$ & $0$ & $-36$
    \tabularnewline
    \hline
    $6$ & $8$ & $0$ & $0$
    \tabularnewline
    \hline
    $6$ & $10$ & $0$ & $-36$
    \tabularnewline
    \hline
    $8$ & $3$ & $0$ & $48$
    \tabularnewline
    \hline
    $8$ & $6$ & $0$ & $-24$
    \tabularnewline
    \hline
    $8$ & $9$ & $0$ & $-24$
    \tabularnewline
    \hline
    $9$ & $5$ & $-\zeta_{3}^{2}$ & $0$
    \tabularnewline
    \hline
    $9$ & $7$ & $-\zeta_{3}$ & $0$
    \tabularnewline
\end{supertabular}

\bigskip\bigskip\bigskip

\topcaption{\small Character values for $G_{26}$\vspace{-10pt}}
\begin{supertabular}{ v{22px} | v{22px} || v{48px} | v{30px} }
    $\deg(\chi)$ & $\occ{\chi}$ & $\chi(c)$ & $\chi(R)$
    \tabularnewline
    \hline
    \hline
    $1$ & $0$ & $1$ & $33$
    \tabularnewline
    \hline
    $1$ & $9$ & $-1$ & $15$
    \tabularnewline
    \hline
    $1$ & $12$ & $\zeta_{3}^{2}$ & $-3$
    \tabularnewline
    \hline
    $1$ & $21$ & $-\zeta_{3}^{2}$ & $-21$
    \tabularnewline
    \hline
    $1$ & $24$ & $\zeta_{3}$ & $-3$
    \tabularnewline
    \hline
    $1$ & $33$ & $-\zeta_{3}$ & $-21$
    \tabularnewline
    \hline
    $2$ & $3$ & $\zeta_{3}$ & $30$
    \tabularnewline
\end{supertabular}

\bigskip\bigskip\bigskip

\topcaption{\small Character values for $G_{27}$\vspace{-10pt}}
\begin{supertabular}{ v{22px} | v{22px} || v{48px} | v{30px} }
    $\deg(\chi)$ & $\occ{\chi}$ & $\chi(c)$ & $\chi(R)$
    \tabularnewline
    \hline
    \hline
    $1$ & $0$ & $1$ & $45$
    \tabularnewline
    \hline
    $1$ & $45$ & $-1$ & $-45$
    \tabularnewline
    \hline
    $3$ & $1$ & $\zeta_{15}^{11} + \zeta_{15}^{14}$ & $45$
    \tabularnewline
    \hline
    $3$ & $5$ & $\zeta_{15}^{7} + \zeta_{15}^{13}$ & $45$
    \tabularnewline
    \hline
    $3$ & $5$ & $\zeta_{15} + \zeta_{15}^{4}$ & $45$
    \tabularnewline
    \hline
    $3$ & $7$ & $\zeta_{15}^{2} + \zeta_{15}^{8}$ & $45$
    \tabularnewline
    \hline
    $3$ & $16$ & $-\zeta_{15}^{11} - \zeta_{15}^{14}$ & $-45$
    \tabularnewline
\end{supertabular}

\bigskip\bigskip\bigskip

\topcaption{\small Character values for $F_4$\vspace{-10pt}}
\begin{supertabular}{ v{22px} | v{22px} || v{48px} | v{30px} }
    $\deg(\chi)$ & $\occ{\chi}$ & $\chi(c)$ & $\chi(R)$
    \tabularnewline
    \hline
    \hline
    $1$ & $0$ & $1$ & $24$
    \tabularnewline
    \hline
    $1$ & $12$ & $1$ & $0$
    \tabularnewline
    \hline
    $1$ & $12$ & $1$ & $0$
    \tabularnewline
    \hline
    $1$ & $24$ & $1$ & $-24$
    \tabularnewline
    \hline
    $2$ & $4$ & $-1$ & $24$
    \tabularnewline
    \hline
    $2$ & $4$ & $-1$ & $24$
    \tabularnewline
    \hline
    $2$ & $16$ & $-1$ & $-24$
    \tabularnewline
    \hline
    $2$ & $16$ & $-1$ & $-24$
    \tabularnewline
    \hline
    $4$ & $1$ & $0$ & $48$
    \tabularnewline
    \hline
    $4$ & $7$ & $0$ & $0$
    \tabularnewline
    \hline
    $4$ & $7$ & $0$ & $0$
    \tabularnewline
    \hline
    $4$ & $8$ & $1$ & $0$
    \tabularnewline
    \hline
    $4$ & $13$ & $0$ & $-48$
    \tabularnewline
    \hline
    $6$ & $6$ & $-1$ & $0$
    \tabularnewline
    \hline
    $6$ & $6$ & $-1$ & $0$
    \tabularnewline
    \hline
    $8$ & $3$ & $0$ & $48$
    \tabularnewline
    \hline
    $8$ & $3$ & $0$ & $48$
    \tabularnewline
    \hline
    $8$ & $9$ & $0$ & $-48$
    \tabularnewline
    \hline
    $8$ & $9$ & $0$ & $-48$
    \tabularnewline
    \hline
    $9$ & $2$ & $0$ & $72$
    \tabularnewline
    \hline
    $9$ & $6$ & $0$ & $0$
    \tabularnewline
    \hline
    $9$ & $6$ & $0$ & $0$
    \tabularnewline
    \hline
    $9$ & $10$ & $0$ & $-72$
    \tabularnewline
    \hline
    $12$ & $4$ & $1$ & $0$
    \tabularnewline
    \hline
    $16$ & $5$ & $0$ & $0$
    \tabularnewline
\end{supertabular}

\bigskip\bigskip\bigskip

\topcaption{\small Character values for $G_{29}$\vspace{-10pt}}
\begin{supertabular}{ v{22px} | v{22px} || v{48px} | v{30px} }
    $\deg(\chi)$ & $\occ{\chi}$ & $\chi(c)$ & $\chi(R)$
    \tabularnewline
    \hline
    \hline
    $1$ & $0$ & $1$ & $40$
    \tabularnewline
    \hline
    $1$ & $40$ & $1$ & $-40$
    \tabularnewline
    \hline
    $4$ & $1$ & $-\zeta_{4}$ & $80$
    \tabularnewline
    \hline
    $4$ & $3$ & $\zeta_{4}$ & $80$
    \tabularnewline
    \hline
    $4$ & $4$ & $-1$ & $80$
    \tabularnewline
    \hline
    $4$ & $21$ & $-\zeta_{4}$ & $-80$
    \tabularnewline
    \hline
    $4$ & $23$ & $\zeta_{4}$ & $-80$
    \tabularnewline
\end{supertabular}

\bigskip\bigskip\bigskip

\topcaption{\small Character values for $H_4$\vspace{-10pt}}
\begin{supertabular}{ v{22px} | v{22px} || v{48px} | v{30px} }
    $\deg(\chi)$ & $\occ{\chi}$ & $\chi(c)$ & $\chi(R)$
    \tabularnewline
    \hline
    \hline
    $1$ & $0$ & $1$ & $60$
    \tabularnewline
    \hline
    $1$ & $60$ & $1$ & $-60$
    \tabularnewline
    \hline
    $4$ & $1$ & $\zeta_{5} + \zeta_{5}^{4}$ & $120$
    \tabularnewline
    \hline
    $4$ & $7$ & $\zeta_{5}^{2} + \zeta_{5}^{3}$ & $120$
    \tabularnewline
    \hline
    $4$ & $31$ & $\zeta_{5} + \zeta_{5}^{4}$ & $-120$
    \tabularnewline
    \hline
    $4$ & $37$ & $\zeta_{5}^{2} + \zeta_{5}^{3}$ & $-120$
    \tabularnewline
    \hline
    $6$ & $12$ & $-\zeta_{5} - \zeta_{5}^{4}$ & $0$
    \tabularnewline
\end{supertabular}

\bigskip\bigskip\bigskip

\topcaption{\small Character values for $G_{32}$\vspace{-10pt}}
\begin{supertabular}{ v{22px} | v{22px} || v{48px} | v{30px} }
    $\deg(\chi)$ & $\occ{\chi}$ & $\chi(c)$ & $\chi(R)$
    \tabularnewline
    \hline
    \hline
    $1$ & $0$ & $1$ & $80$
    \tabularnewline
    \hline
    $1$ & $40$ & $\zeta_{3}$ & $-40$
    \tabularnewline
    \hline
    $1$ & $80$ & $\zeta_{3}^{2}$ & $-40$
    \tabularnewline
    \hline
    $4$ & $1$ & $\zeta_{3}$ & $200$
    \tabularnewline
    \hline
    $4$ & $11$ & $\zeta_{3}^{2}$ & $200$
    \tabularnewline
    \hline
    $4$ & $21$ & $1$ & $-40$
    \tabularnewline
    \hline
    $4$ & $41$ & $\zeta_{3}^{2}$ & $-160$
    \tabularnewline
\end{supertabular}

\bigskip\bigskip\bigskip

\topcaption{\small Character values for $G_{33}$\vspace{-10pt}}
\begin{supertabular}{ v{22px} | v{22px} || v{48px} | v{30px} }
    $\deg(\chi)$ & $\occ{\chi}$ & $\chi(c)$ & $\chi(R)$
    \tabularnewline
    \hline
    \hline
    $1$ & $0$ & $1$ & $45$
    \tabularnewline
    \hline
    $1$ & $45$ & $-1$ & $-45$
    \tabularnewline
    \hline
    $5$ & $1$ & $\zeta_{3}^{2}$ & $135$
    \tabularnewline
    \hline
    $5$ & $3$ & $\zeta_{3}$ & $135$
    \tabularnewline
    \hline
    $5$ & $28$ & $-\zeta_{3}^{2}$ & $-135$
    \tabularnewline
    \hline
    $5$ & $30$ & $-\zeta_{3}$ & $-135$
    \tabularnewline
    \hline
    $6$ & $5$ & $0$ & $90$
    \tabularnewline
\end{supertabular}

\bigskip\bigskip\bigskip

\topcaption{\small Character values for $G_{34}$\vspace{-10pt}}
\begin{supertabular}{ v{22px} | v{22px} || v{48px} | v{30px} }
    $\deg(\chi)$ & $\occ{\chi}$ & $\chi(c)$ & $\chi(R)$
    \tabularnewline
    \hline
    \hline
    $1$ & $0$ & $1$ & $126$
    \tabularnewline
    \hline
    $1$ & $126$ & $1$ & $-126$
    \tabularnewline
    \hline
    $6$ & $1$ & $\zeta_{3}^{2}$ & $504$
    \tabularnewline
    \hline
    $6$ & $5$ & $\zeta_{3}$ & $504$
    \tabularnewline
    \hline
    $6$ & $85$ & $\zeta_{3}^{2}$ & $-504$
    \tabularnewline
    \hline
    $6$ & $89$ & $\zeta_{3}$ & $-504$
    \tabularnewline
    \hline
    $15$ & $14$ & $\zeta_{3}$ & $630$
    \tabularnewline
\end{supertabular}

\bigskip\bigskip\bigskip

\topcaption{\small Character values for $E6$}
\begin{supertabular}{ v{22px} | v{22px} || v{48px} | v{30px} }
    $\deg(\chi)$ & $\occ{\chi}$ & $\chi(c)$ & $\chi(R)$
    \tabularnewline
    \hline
    \hline
    $1$ & $0$ & $1$ & $36$
    \tabularnewline
    \hline
    $1$ & $36$ & $1$ & $-36$
    \tabularnewline
    \hline
    $6$ & $1$ & $-1$ & $144$
    \tabularnewline
    \hline
    $6$ & $25$ & $-1$ & $-144$
    \tabularnewline
    \hline
    $10$ & $9$ & $-1$ & $0$
    \tabularnewline
    \hline
    $15$ & $4$ & $-1$ & $180$
    \tabularnewline
    \hline
    $15$ & $5$ & $0$ & $180$
    \tabularnewline
    \hline
    $15$ & $16$ & $-1$ & $-180$
    \tabularnewline
    \hline
    $15$ & $17$ & $0$ & $-180$
    \tabularnewline
    \hline
    $20$ & $2$ & $0$ & $360$
    \tabularnewline
    \hline
    $20$ & $10$ & $1$ & $0$
    \tabularnewline
    \hline
    $20$ & $20$ & $0$ & $-360$
    \tabularnewline
    \hline
    $24$ & $6$ & $0$ & $144$
    \tabularnewline
    \hline
    $24$ & $12$ & $0$ & $-144$
    \tabularnewline
    \hline
    $30$ & $3$ & $1$ & $360$
    \tabularnewline
    \hline
    $30$ & $15$ & $1$ & $-360$
    \tabularnewline
    \hline
    $60$ & $5$ & $0$ & $360$
    \tabularnewline
    \hline
    $60$ & $8$ & $1$ & $0$
    \tabularnewline
    \hline
    $60$ & $11$ & $0$ & $-360$
    \tabularnewline
    \hline
    $64$ & $4$ & $0$ & $576$
    \tabularnewline
    \hline
    $64$ & $13$ & $0$ & $-576$
    \tabularnewline
    \hline
    $80$ & $7$ & $0$ & $0$
    \tabularnewline
    \hline
    $81$ & $6$ & $0$ & $324$
    \tabularnewline
    \hline
    $81$ & $10$ & $0$ & $-324$
    \tabularnewline
    \hline
    $90$ & $8$ & $-1$ & $0$
    \tabularnewline
\end{supertabular}

\bigskip\bigskip\bigskip

\topcaption{\small Character values for $E7$}
\begin{supertabular}{ v{22px} | v{22px} || v{48px} | v{30px} }
    $\deg(\chi)$ & $\occ{\chi}$ & $\chi(c)$ & $\chi(R)$
    \tabularnewline
    \hline
    \hline
    $1$ & $0$ & $1$ & $63$
    \tabularnewline
    \hline
    $1$ & $63$ & $-1$ & $-63$
    \tabularnewline
    \hline
    $7$ & $1$ & $-1$ & $315$
    \tabularnewline
    \hline
    $7$ & $46$ & $1$ & $-315$
    \tabularnewline
    \hline
    $15$ & $7$ & $0$ & $315$
    \tabularnewline
    \hline
    $15$ & $28$ & $0$ & $-315$
    \tabularnewline
    \hline
    $21$ & $3$ & $0$ & $693$
    \tabularnewline
\end{supertabular}

\bigskip\bigskip\bigskip

\topcaption{\small Character values for $E8$}\label{tab:last}
\begin{supertabular}{ v{22px} | v{22px} || v{48px} | v{30px} }
    $\deg(\chi)$ & $\occ{\chi}$ & $\chi(c)$ & $\chi(R)$
    \tabularnewline
    \hline
    \hline
    $1$ & $0$ & $1$ & $120$
    \tabularnewline
    \hline
    $1$ & $120$ & $1$ & $-120$
    \tabularnewline
    \hline
    $8$ & $1$ & $-1$ & $720$
    \tabularnewline
    \hline
    $8$ & $91$ & $-1$ & $-720$
    \tabularnewline
    \hline
    $28$ & $8$ & $0$ & $1680$
    \tabularnewline
    \hline
    $28$ & $68$ & $0$ & $-1680$
    \tabularnewline
    \hline
    $35$ & $2$ & $0$ & $2520$
    \tabularnewline
\end{supertabular}

\onecolumn

\normalsize
\bibliographystyle{alpha}
\bibliography{ChapuyStump2012.bib}

\end{document}